\documentclass[reqno]{amsart}
\usepackage{amssymb,latexsym,amsmath,amsthm,enumerate,amsbsy}
\usepackage[mathscr]{eucal}
\usepackage{framed,color,graphicx}
\usepackage{mathrsfs}
\usepackage[all]{xy}
\usepackage{tikz}
\usetikzlibrary{positioning,decorations.pathreplacing,patterns,decorations.pathmorphing}
\tikzset{%
element/.style={draw, shape=circle, fill=white, inner sep=1.4pt}
}

\DeclareSymbolFont{bbold}{U}{bbold}{m}{n}
\DeclareSymbolFontAlphabet{\mathbbold}{bbold}

\theoremstyle{plain}
\newtheorem{thm}{Theorem}[section]
\newtheorem{lem}[thm]{Lemma}
\newtheorem{cor}[thm]{Corollary}
\newtheorem{pro}[thm]{Proposition}
\newtheorem{example}[thm]{Example}

\theoremstyle{definition}

\newtheorem{remark}[thm]{Remark}

\newcommand{\KG}{\operatorname{KG}}

\newcommand{\up}[1]{\textup{#1}}

\newcommand{\bp}{\mathbf{p}}
\newcommand{\bq}{\mathbf{q}}

\newcommand{\bt}{\mathbf{t}}
\newcommand{\bu}{\mathbf{u}}
\newcommand{\bv}{\mathbf{v}}
\newcommand{\bw}{\mathbf{w}}

\begin{document}

\title[The finite basis problem]
{The finite basis problem for additively idempotent semirings that relate to $S_7$}

\author{Zidong Gao}
\address{School of Mathematics, Northwest University, Xi'an, 710127, Shaanxi, P.R. China}
\email{zidonggao@yeah.net}

\author{Marcel Jackson}
\address{Department of Mathematical and Physical Sciences\\ La Trobe University\\ Victoria  3086\\
Australia} \email{m.g.jackson@latrobe.edu.au}

\author{Miaomiao Ren}
\address{School of Mathematics, Northwest University, Xi'an, 710127, Shaanxi, P.R. China}
\email{miaomiaoren@yeah.net}

\author{Xianzhong Zhao}
\address{School of Mathematics, Northwest University, Xi'an, 710127, Shaanxi, P.R. China}
\email{zhaoxz@nwu.edu.cn}

\thanks{Miaomiao Ren, corresponding author, is supported by National Natural Science Foundation of China (12371024).
}

\begin{abstract}
The $3$-element additively idempotent semiring $S_7$
is a nonfinitely based algebra of the smallest possible order.
In this paper we study the finite basis problem for some additively idempotent semirings that relate to $S_7$.
We present a sufficient condition under which an additively idempotent semiring variety is nonfinitely based and as applications, show that some additively idempotent semiring varieties that contain $S_7$ are also nonfinitely based.
We then consider the subdirectly irreducible members of the variety $\mathsf{V}(S_7)$ generated by $S_7$.  We show that $\mathsf{V}(S_7)$ contains exactly 6 finitely based subvarieties, all of which sit at the base of the subvariety lattice, then invoke results from the homomorphism theory of Kneser graphs to verify that $\mathsf{V}(S_7)$ contains a continuum of subvarieties. \end{abstract}

\maketitle

\section{Introduction and preliminaries}\label{sec:intro}
A \emph{variety} is a class of algebras that is closed under taking
subalgebras, homomorphic images and arbitrary direct products.
Equivalently, by Birkhoff's well known theorem, a class of algebras is a variety
if and only if it is an \emph{equational class}, that is,
the class of all algebras that satisfy some set of identities (also known as equations).
Let $\mathcal{V}$ be a variety.
One of the most important considerations on $\mathcal{V}$ is whether
it is \emph{finitely based}, that is, if it can be defined by finitely many identities.
The variety $\mathcal{V}$ is \emph{nonfinitely based} if it is not finitely based.
An algebra $A$ is called finitely based (nonfinitely based) if the variety $\mathsf{V}(A)$ generated by $A$ is finitely based (nonfinitely based).

In 1951 Lyndon \cite{lyn1} showed that every $2$-element algebra is finitely based.
Subsequently, Lyndon \cite{lyn2} provided a 7-element nonfinitely based groupoid,
which is the first example of nonfinitely based finite algebras.
In 1965 Murski\v{\i} \cite{mur} presented a 3-element nonfinitely based groupoid,
with further 3-element examples subsequently identified by Je\v{z}ek~\cite{jez}.
In 2022 Jackson, Ren and Zhao \cite{jrz} showed that there is a 3-element semiring that is additively idempotent and with commutative multiplication that is also nonfinitely based.
The Cayley tables for addition and multiplication of this semiring, denoted by $S_7$, can be found in Table~\ref{tb24111401}.
The semiring $S_7$ has the following remarkable property (see~\cite{jrz}).
Up to isomorphism, it is the only nonfinitely based additively idempotent semiring of order at most three.
Also, it can infect the nonfinitely based property to many other finite additively idempotent semirings, though it remains unknown if all finite additively idempotent semirings whose variety contains $S_7$ are nonfinitely based.
The present paper follows this line of investigation, by exploring finite basis problem and other variety-theoretic properties of  some additively idempotent semirings that relate to~$S_7$.
The infectiousness of the nonfinite basis property of $S_7$ is pushed to some new examples by way of a new nonfinite basis result.  We then explore the lattice of subvarieties of $\mathsf{V}(S_7)$, eventually showing that this has size  $2^{\aleph_0}$, and that all except six of these are nonfinitely based.

\begin{table}[ht]
\caption{The Cayley tables of $S_7$} \label{tb24111401}
\begin{tabular}{c|ccc}
$+$      &$\infty$&$a$&$1$\\
\hline
$\infty$ &$\infty$&$\infty$&$\infty$\\
$a$      &$\infty$&$a$&$\infty$\\
$1$      &$\infty$&$\infty$&$1$\\
\end{tabular}\qquad
\begin{tabular}{c|ccc}
$\cdot$  &$\infty$&$a$&$1$\\
\hline
$\infty$ &$\infty$&$\infty$&$\infty$\\
$a$      &$\infty$&$\infty$&$a$\\
$1$      &$\infty$&$a$&$1$\\
\end{tabular}
\end{table}

An \emph{additively idempotent semiring} (ai-semiring for short) is an algebra $(S, +, \cdot)$ with two binary operations $+$ and $\cdot$
such that the additive reduct $(S, +)$ is a commutative idempotent semigroup,
the multiplicative reduct $(S, \cdot)$ is a semigroup and $S$ satisfies the distributive laws
\[
(x+y)z\approx xy+xz,\quad x(y+z)\approx xy+xz.
\]
The set of all natural numbers forms an ai-semiring under the usual addition and multiplication.
One can easily find many other examples of ai-semirings in almost all branches of mathematics.
Such algebras have played important roles in several branches of mathematics
such as algebraic geometry \cite{cc}, tropical geometry \cite{ms}, information science \cite{gl} and
theoretical computer science \cite{go}.
In the last two decades, the finite basis problem for ai-semirings have been intensively studied
and well developed, for example,
see \cite{dol07, dgv, gpz05, jrz, pas05, rlzc, rjzl, sr, vol21, wrz, yrzs, zrc}.

A \emph{flat semiring} is an ai-semiring such that its multiplicative reduct has a zero element $0$
and $a+b=0$ for all distinct elements $a$ and $b$ of $S$.
Jackson et al.~\cite[Lemma 2.2]{jrz} proved that a semigroup with the zero element $0$
becomes a flat semiring with the top element $0$ if and only if it is
$0$-cancellative, that is, $ab=ac\neq0$ implies $b=c$ and $ba=ca\neq0$ implies $b=c$
for all $a, b, c\in S$.
More generally, for any partial algebra $P$ we may add a new element $\infty$, and a new operation $+$ making $P\cup\{\infty\}$ into a flat semilattice with top $\infty$, then create  a total algebra  on $P\cup\{\infty\}$ by letting $\infty$ be the output at otherwise undefined inputs.
This total algebra is known as the \emph{flat extension of $P$} and denoted $\flat(P)$.
When $P$ is a partial groupoid,
then $\flat(P)$ is a flat semiring exactly when $(xy)z$ is defined in $P$ if and only if $x(yz)$ is defined
and $(xy)z=x(yz)$ if they are both defined,
and the obvious analogue of $0$-cancellativity holds: $ab$ and $ac$ are defined and equal implies $b=c$, and dually for $ba$ and $ca$.

Let ${\bf F}$ denote the variety generated by all flat semirings.
The following result, due to Jackson et al.~\cite[Lemma 2.1]{jrz} (see also~\cite{jac:flat}), answered the finite basis problem for ${\bf F}$
and described the subdirectly irreducible members of ${\bf F}$.
\begin{lem}\label{lem24121301}
The variety ${\bf F}$ is finitely based and each subdirectly irreducible member of ${\bf F}$ is a flat semiring.
\end{lem}

The following algebras form an important class of flat semirings.
Let $W$ be a nonempty subset of a free commutative semigroup,
and let $S_c(W)$ denote the set of all nonempty subwords of words in
$W$ together with a new symbol $0$. Define a binary operation $\cdot$ on $S_c(W)$ by the rule
\begin{equation*}
\bu\cdot \bv=
\begin{cases}
\bu\bv& \text{if }~\bu\bv\in S_c(W)\setminus \{0\}, \\
0& \text{otherwise.}
\end{cases}
\end{equation*}
Then $(S_c(W), \cdot)$ forms a commutative semigroup with zero element $0$.
It is easy to verify that $(S_c(W), \cdot)$ is $0$-cancellative and so $S_c(W)$ becomes a flat semiring.
In particular, if~$W$ consists of a single word $\bw$ we shall write $S_c(W)$ as $S_c(\bw)$.
If we allow the empty word in this construction, then there is a multiplicative identity element and the semigroup reduct is a monoid.
The notation $M_c(W)$ is used in this case.
If $a$ is a letter, then $M_c(a)$ is isomorphic to~$S_7$ and $S_c(a)$ can be embedded into~$S_7$.
If $1$ denotes the empty word, then $M_c(1)$ can be embedded into~$S_7$.
Note that $S_c(a)$ and $M_c(1)$ are denoted by $T_2$ and $M_2$ in Shao and Ren~\cite{sr}, respectively.

The non-commutative version of the $M_c(W)$ has also been explored: the subscript $c$ is dropped in the notations just given, giving $M(W)$ (where $W$ is a set of words from the free semigroup).  Examples of the form $M(W)$ have been particularly heavily explored in semigroup context, but have also appeared in semiring settings, such as Ren et al.~\cite{rjzl}.
Note that the single letter word is degenerate enough that $M_c(a)=M(a)$, so $S_7$ has relevance to both the commutative and noncommutative examples.

Let $X^+$ denote the free semigroup over a countably infinite set $X$ of variables.
By distributivity, all ai-semiring terms over $X$ are finite sums of words in $X^+$.
An \emph{ai-semiring identity} over $X$ is a formal expression of the form
\[
\bu\approx \bv,
\]
where $\bu$ and $\bv$ are ai-semiring terms over $X$.
Let $S$ be an ai-semiring, and let $\bu\approx \bv$ be an ai-semiring identity over $\{x_1, x_2, \ldots, x_n\}$.
Then $S$ \emph{satisfies} $\bu\approx \bv$ if
$\bu(a_1, a_2, \ldots, a_n)=\bv(a_1, a_2, \ldots, a_n)$ for all $a_1, a_2, \ldots, a_n\in S$, where $\bu(a_1, a_2, \ldots, a_n)$ denotes the result of evaluating $\bu$ in $S$ under the assignment $x_i\mapsto a_i$, and similarly for $\bv(a_1, a_2, \ldots, a_n)$.

Let $\bp$ be a word and $x$ a letter. Then $occ(x, \bp)$ denotes the number of occurrences of $x$ in $\bp$.
Let $\bu$ be an ai-semiring term. Then
\begin{itemize}
\item $c(\bu)$ denotes the set of variables that occur in $\bu$;

\item $\delta(\bu)$ denotes the set of nonempty subsets $Z$ of $c(\bu)$ such that for every $\bp \in \bu$,
$Z\cap c(\bp)$ is a singleton and $occ(x, \bp)=1$ if $\{x\}=Z\cap c(\bp)$.
\end{itemize}
The following result, which provides the solution of the equational problem for $S_7$,
is due to Jackson et al.~\cite[Proposition 5.5]{jrz}.
\begin{lem}\label{lemma24221510}
Let $\bu\approx \bv$ be an ai-semiring identity. Then $\bu\approx \bv$ holds in $S_7$
if and only if $c(\bu)=c(\bv)$ and $\delta(\bu)=\delta(\bv)$.
\end{lem}

Let $S_7^0$ denote the ai-semiring that is obtained from $S_7$ by adding an extra element $0$,
where
\[
(\forall a\in S_7^0) \quad a+0=a,~ a0=0a=0.
\]
It is easy to see that $S_7$ is a subalgebra of $S_7^0$.
Jackson et al.~\cite{jrz} asked whether $S_7^0$ is finitely based.
Wu et al.~\cite{wrz} used a syntactic approach to show that every subvariety of $\mathsf{V}(S_7^0)$
that contains $S_7$ and the $2$-element distributive lattice is nonfinitely based.
In particular, $S_7^0$ is nonfinitely based.
In Section~\ref{sec:NFB} we shall use a different method to extend the above result.
The following result, which concerns the equational problem of $S^0_7$,
can be found in \cite[Proposition 2.2]{wrz}.
\begin{lem}\label{lemma24221511}
Let $\bu\approx \bu+\bq$ be an ai-semiring identity such that $\bu=\bu_1+\cdots+\bu_n$,
where $\bq, \bu_i\in X^+$, $1\leq i\leq n$.
If $c(\bu)=c(\bq)$ and $\delta(\bu)=\varnothing$,
then $\bu\approx \bu+\bq$ is satisfied by $S^0_7$.
\end{lem}

Next, we recall some knowledge on hypergraph semirings that are introduced by Jackson et al.~\cite{jrz}.
A \emph{$3$-uniform hypergraph} $\mathbb{H}$ is a pair $\langle V, E\rangle$,
where $E$ is a family of nonempty subsets of a set $V$, and $|e|=3$ for all $e\in E$.
Each element of $V$ is a \emph{vertex} of $\mathbb{H}$,
and each element of $E$ is a \emph{hyperedge} of $\mathbb{H}$.
Let $\mathbb{H}$ be a $3$-uniform hypergraph.
Then $\mathbb{H}$ is \emph{$2$-colourable} if there exists a mapping $\varphi: V \to \{0, 1\}$
such that $|\varphi(e)|=2$ for all $e\in E$, and
$\mathbb{H}$ is \emph{$2$-in-$3$ satisfiable} if there exists a mapping
$\varphi: V \to \{0, 1\}$ such that $|\varphi^{-1}(0) \cap e|=1$ for all $e\in E$.
It is easy to see that if ${\mathbb H}$ is $2$-in-$3$ satisfiable, then it must be $2$-colourable.
Colourings relate to homomorphisms, which we now describe.
A \emph{homomorphism} between hypergraphs~$\mathbb{G}$ and $\mathbb{H}$, is a function $\phi$ from the vertices of $\mathbb{G}$ to the vertices of $\mathbb{H}$ such that every hyperedge  of $\mathbb{G}$ maps to a hyperedge of $\mathbb{H}$.
An \emph{isomorphism} between hypergraphs is of course a homomorphism that is bijective on vertices and on hyperedges.

A \emph{cycle} of $\mathbb{H}$ is a sequence
$v_1, e_1, v_2, e_2, \ldots, v_{n}, e_{n}$
alternating between vertices and hyperedges, with no repeats,
such that $v_1\in e_1\cap e_{n}$ and $v_{i+1}\in e_i\cap e_{i+1}$, $1\leq i < n$.
The number $n$ is the \emph{length} of this cycle.
A \emph{girth} of $\mathbb{H}$ is the length of its shortest cycles.
Let $\ell\geq 2$ be an integer.  Then
\cite[Theorem 2.7]{hj} tells us that there is a $3$-uniform hypergraph
that is not $2$-colourable and whose girth is great than or equal to $\ell$.

Let $\mathbb{H}=\langle V, E\rangle$ be a $3$-uniform hypergraph that has no isolated vertices
and whose girth is greater than or equal to $5$.
A \emph{hypergraph semiring} $S_\mathbb{H}$ defined by $\mathbb{H}$ is a flat semiring generated by a copy
$\{\mathbf{a}_v \mid v\in V\}$ of $V$
along with a special element~$\infty$ and subject to the following rules:
\begin{itemize}
\item[$(1)$] $\infty$ is the multiplicative zero element;

\item[$(2)$] $\mathbf{a}_u\mathbf{a}_v=\mathbf{a}_v\mathbf{a}_u$ for all $u, v\in V$;

\item[$(3)$] $\mathbf{a}_u\mathbf{a}_v=\infty$ if $\{u, v\}$ is not a $2$-element subset of a hyperedge in $E$;

\item[$(4)$] $\mathbf{a}_{u_1}\mathbf{a}_{u_2}\mathbf{a}_{u_3}=\mathbf{a}_{v_1}\mathbf{a}_{v_2}\mathbf{a}_{v_3}$
if $\{u_1, u_2, u_3\}, \{v_1, v_2, v_3\}\in E$;

\item[$(5)$] $\mathbf{a}_{u_1}\mathbf{a}_{u_2}=\mathbf{a}_{v_1}\mathbf{a}_{v_2}$
if $\{u_1, u_2, \omega\}, \{v_1, v_2, \omega\} \in E$ for some $\omega\in V$.
\end{itemize}
We let $\mathbf{a}$ denote the element arising in item (4).
In Section~\ref{sec:blockhypergraph} we revisit this construction allowing a broader family of hypergraphs.

For other notations and terminology used in this paper, the reader is referred to
Jackson et al.~\cite{jrz} and Ren et al.~\cite{rjzl} for background on semirings,
and to Burris and Sankappanavar~\cite{bs} for information concerning universal algebra.
We shall assume that the reader is familiar with the basic results in these areas.


\section{A sufficient condition for the nonfinitely based property}\label{sec:NFB}
In this section we use hypergraph semirings to establish a sufficient condition
under which an ai-semiring variety is nonfinitely based.
As applications, we show that some ai-semiring varieties that contain $S_7$ are nonfinitely based.

Let $n\geq 2$ be a natural number, and let $\mathbb{H}=\langle V, E\rangle$
be a $3$-uniform hypergraph 
From $\mathbb{H}$ one can define some ai-semiring terms.
Let $\{x_v \mid v \in V_n\}$ be a set of variables that is a copy of $V$.
We shall use ${\bf t}_{\mathbb{H}}$ to denote the ai-semiring term
\[
\sum_{\{v_1, v_2, v_3\}\in E} x_{v_1}x_{v_2}x_{v_3},
\]
and $\bq_{_{\mathbb{H}}}$ to denote the word
\[
\prod_{v \in V}x_v,
\]
where it will suffice to adopt any fixed order on the variables in this product.
A product $x_{v_1}x_{v_2}x_{v_3}$ where $\{v_1, v_2, v_3\}\in E$ will be referred to as a \emph{hyperedge product} (for $\mathbb{H}$); every hyperedge gives rise to 6 different hyperedge products due to the fact that the hyperedge $\{v_1, v_2, v_3\}$ is unordered, so all~6 permutations of the variables satisfy the definition.  Let ${\bf w}$ denote any ai-semiring term in the variables $\{x_v \mid v \in V\}$ for which the evaluation $x_v\mapsto \mathbf{a}_v$ leads to $\bw$ \emph{not} taking the value $\mathbf{a}$.  We refer to such a term as a \emph{non-hyperedge term}.
Technically, non-hyperedge terms might depend on the hypergraph being considered, however there are many choices that apply to every hypergraph, simply because they involve products that are too long, or too short.
The following example gives three examples of such terms that will be used later in the paper.

\begin{example}
The following are non-hypergraph terms for every hypergraph, where the third item applies only in the case where $\mathbb{H}$ has more than~$3$ vertices\up:
\begin{align}
{\bf w} & := \bt_{\mathbb{H}}^2; \label{id112501}\\
{\bf w}  & := x_v \quad \text{\up(for some vertex $v$\up)};    \label{id112502}\\
{\bf w}& := \bq_{_{\mathbb{H}}} \label{id112503}
\end{align}
\end{example}

We  fix a family of $3$-uniform hypergraphs $\mathbb{H}_n=\langle V_n, E_n\rangle$ indexed by $n\in\mathbb{N}$
where $\mathbb{H}_n$ is not $2$-colourable and has girth greater than $3\binom{3n}{2}$; as explained in~\cite{jrz}, such a family exists by way of Erd\H{o}s and Hajnal~\cite{erdhaj}.  In this context, ${\bf w}_n$ will be used for a non-hypergraph term relative to $\mathbb{H}_n$.
\begin{thm}\label{thm24112501}
Let $\mathcal{V}$ be an ai-semiring variety that contains $S_c(abc)$,
and for $n\geq 2$ let ${\bf w}_n$ be a non-hyperedge term for $\mathbb{H}_n$.
If $\mathcal{V}$ satisfies the identity
\begin{equation}
\bt_{\mathbb{H}_n}  \approx \bt_{\mathbb{H}_n}+{\bf w}_n\label{id112500}
\end{equation}
for all $n\geq 2$,
then $\mathcal{V}$ has no finite basis for its equational theory.
\end{thm}
\begin{proof}
The strategy of the proof is as follows.
We shall show that for every $n \geq 2$, the set
of all $n$-variable identities of $\mathcal{V}$ does not form a basis for the equational theory of $\mathcal{V}$.
For this it is enough to prove that for every $n \geq 2$,
$S_{\mathbb{H}_n}$ does not lie in~$\mathcal{V}$,
but all $n$-generated subalgebras of $S_{\mathbb{H}_n}$ do lie in~$\mathcal{V}$.

Let $n\geq 2$.
From the proof of \cite[Theorem 4.9]{jrz}
we know that each $n$-generated subalgebra of $S_{\mathbb{H}_n}$
is a member of $\mathsf{V}(S_c(abc))$ and so it lies in $\mathcal{V}$.
Now let us consider the substitution $\varphi: X \to S_{\mathbb{H}_n}$
defined by $\varphi(x_v)=\mathbf{a}_v$ for all $v\in V$.
It is easy to see that $\varphi(t_{\mathbb{H}_n})=\mathbf{a}$ and by the definition of ${\bf w}_n$
we also have that $\varphi(\bw_n)\neq \mathbf{a}$.
This shows that $S_{\mathbb{H}_n}$ does not satisfy the identity~\eqref{id112500}.
It follows that $S_{\mathbb{H}_n}$ is not a member of~$\mathcal{V}$.
Thus $\mathcal{V}$ is nonfinitely based as required.
\end{proof}
\begin{remark}
The choice of $\bw_n:=\bq_{_{\mathbb{H}_n}}$ in \eqref{id112500} is a hybrid of that used in~\cite{jrz} and~\cite{wrz} and was also recently employed in the context of HSI algebras by Alsulami and Jackson~\cite{alsjac}.
\end{remark}

\begin{pro}\label{pro24112701}
Let ${\mathbb H}=\langle V, E\rangle$ be a $3$-uniform hypergraph. Then
$\delta(\bt_{\mathbb H})\neq \varnothing$ if and only if $\mathbb H$ is $2$-in-$3$ satisfiable.
\end{pro}

\begin{proof}
If $\delta(t_{\mathbb H}) \neq \varnothing$ then there exists $Z\subseteq c(\bt_{\mathbb H})$ such that
$|Z\cap \{x_{v_1},\ldots,x_{v_k}\}|=1$ for all $\{v_1,\ldots,v_k\}\in E$.
Consider the mapping $\varphi: V \to \{0, 1\}$ defined by
\[
\varphi(v)=\begin{cases}
  0,& x_v\in Z\\
  1,& x_v\notin Z.
\end{cases}
\]
It is easy to see that $|\varphi^{-1}(0)\cap e|=1$ for all $e\in E$.
So $\mathbb H$ is $2$-in-$3$ satisfiable.

Conversely, if $\mathbb H$ is $2$-in-$3$ satisfiable,
then there exists a mapping $\varphi: V \to \{0, 1\}$ such that $|\varphi^{-1}(0)\cap e|=1$ for all $e\in E$.
Let $Z$ denote the set $\{x_v \mid \varphi(v)=0\}$.
It is a routine matter to verify that $Z$ is an element of $\delta(\bt_{\mathbb H})$. Thus $\delta(\bt_{\mathbb H})\neq \varnothing$ as required.
\end{proof}

Let $\mathcal{V}_1$ and $\mathcal{V}_2$ be ai-semiring varieties such that $\mathcal{V}_1$ is contained in $\mathcal{V}_2$.
Then the interval $[\mathcal{V}_1, \mathcal{V}_2]$ denotes the class of all subvarieties of $\mathcal{V}_2$ that contain $\mathcal{V}_1$.
From \cite[Proposition 2.6]{jrz} we know that $\mathsf{V}(S_7)$ contains $S_c(abc)$.

\begin{cor}\label{cor24112501}
Every variety in $[\mathsf{V}(S_c(abc)), \mathsf{V}(S_7^0)]$ is nonfinitely based.
\end{cor}
\begin{proof}
Let $\mathcal{V}$ be an arbitrary variety in $[\mathsf{V}(S_c(abc)), \mathsf{V}(S_7^0)]$.
By Theorem \ref{thm24112501} it is enough to prove that $\mathcal{V}$ satisfies
the identity \eqref{id112500} for all $n\geq2$ in the case of the choice of $\bw_n$ from \eqref{id112503}.
We know that $\mathbb{H}_n$ is not $2$-colourable.
By Proposition \ref{pro24112701} one can deduce that $\delta(\bt_{\mathbb{H}_n})$ is empty.
Since $c(\bt_{\mathbb{H}_n})=c(\bq_{_{\mathbb{H}_n}})$, it follows from
Lemma~\ref{lemma24221511} that the identity \eqref{id112500}
is satisfied by $S_7^0$ and so does hold in $\mathcal{V}$.
\end{proof}

\begin{remark}
By using completely different method,
Corollary \ref{cor24112501} extends \cite[Theorem 2.3]{wrz}, which asserts that
every subvariety of $\mathsf{V}(S_7^0)$ that contains $S_7$ and the $2$-element distributive lattice is nonfinitely based.
\end{remark}

In the remainder we shall apply Theorem \ref{thm24112501} to answer the finite basis problem for some $4$-element ai-semirings.
Following the notations in \cite{rlzc, yrzs},
these algebras are denoted by $S_{(4, k)}$, where $k=84, 94, 117, 123, 173, 282, 359$.
We assume that the carrier set of each of these semirings is $\{\infty, a, 1, b\}$.
Their Cayley tables for addition and multiplication are listed in Table \ref{tab24112801}.
\begin{table}[ht]
\caption{Cayley tables of some 4-element ai-semirings}\label{tab24112801}
\setlength{\tabcolsep}{2.6pt}
\begin{tabular}{ccccccc}
\hline
Semiring&  $+$ &  $\cdot$ &Semiring& $+$ &  $\cdot$& \\
\hline
$S_{(4, 84)}$
&
\begin{tabular}{cccc}
$\infty$&$\infty$&$\infty$&$\infty$\\
      $\infty$&$a$&$\infty$&$\infty$\\
      $\infty$&$\infty$&$1$&$b$\\
     $\infty$&$\infty$&$b$&$b$\\
\end{tabular}
&
\begin{tabular}{cccc}
$\infty$&$\infty$&$\infty$&$\infty$\\
     $\infty$&$\infty$&$a$&$\infty$\\
     $\infty$&$a$&$1$&$\infty$\\
     $\infty$&$\infty$&$\infty$&$\infty$\\
\end{tabular}
&
$S_{(4, 94)}$
&
\begin{tabular}{cccc}
$\infty$&$\infty$&$\infty$&$\infty$\\
      $\infty$&$a$&$\infty$&$\infty$\\
      $\infty$&$\infty$&$1$&$b$\\
     $\infty$&$\infty$&$b$&$b$\\
\end{tabular}
&
\begin{tabular}{cccc}
$\infty$&$\infty$&$\infty$&$\infty$\\
     $\infty$&$\infty$&$a$&$\infty$\\
     $\infty$&$a$&$1$&$\infty$\\
     $\infty$&$b$&$\infty$&$\infty$\\
\end{tabular}
\\ \hline
$S_{(4, 117)}$
&
\begin{tabular}{cccc}
$\infty$&$\infty$&$\infty$&$\infty$\\
      $\infty$&$a$&$\infty$&$\infty$\\
      $\infty$&$\infty$&$1$&$b$\\
     $\infty$&$\infty$&$b$&$b$\\
\end{tabular}
&
\begin{tabular}{cccc}
$\infty$&$\infty$&$\infty$&$\infty$\\
     $\infty$&$\infty$&$a$&$b$\\
     $\infty$&$a$&$1$&$\infty$\\
     $\infty$&$\infty$&$\infty$&$\infty$\\
\end{tabular}
&
$S_{(4, 123)}$
&
\begin{tabular}{cccc}
$\infty$&$\infty$&$\infty$&$\infty$\\
      $\infty$&$a$&$\infty$&$\infty$\\
      $\infty$&$\infty$&$1$&$b$\\
     $\infty$&$\infty$&$b$&$b$\\
\end{tabular}
&
\begin{tabular}{cccc}
$\infty$&$\infty$&$\infty$&$\infty$\\
     $\infty$&$\infty$&$a$&$b$\\
     $\infty$&$a$&$1$&$\infty$\\
     $\infty$&$b$&$\infty$&$\infty$\\
\end{tabular}
\\ \hline
$S_{(4, 173)}$
&
\begin{tabular}{cccc}
$\infty$&$\infty$&$\infty$&$\infty$\\
      $\infty$&$a$&$\infty$&$\infty$\\
      $\infty$&$\infty$&$1$&$b$\\
     $\infty$&$\infty$&$b$&$b$\\
\end{tabular}
&
\begin{tabular}{cccc}
$\infty$&$\infty$&$\infty$&$\infty$\\
     $\infty$&$\infty$&$a$&$a$\\
     $\infty$&$a$&$1$&$1$\\
     $\infty$&$a$&$1$&$1$\\
\end{tabular}
&
$S_{(4, 282)}$
&
\begin{tabular}{cccc}
$\infty$&$\infty$&$\infty$&~$b$\\
      $\infty$&$a$&$\infty$&~$b$\\
      $\infty$&$\infty$&$1$&~$b$\\
     $b$&\,$b$&\,$b$&~$b$\\
\end{tabular}
&
\begin{tabular}{cccc}
$\infty$&$\infty$&$\infty$&$\infty$\\
     $\infty$&$\infty$&$a$&$\infty$\\
     $\infty$&$a$&$1$&$\infty$\\
     $\infty$&$\infty$&$\infty$&$\infty$\\
\end{tabular}
\\ \hline
$S_{(4, 359)}$
&
\begin{tabular}{cccc}
$\infty$&$\infty$&$\infty$&~$b$\\
      $\infty$&$a$&$\infty$&~$b$\\
      $\infty$&$\infty$&$1$&~$b$\\
     $b$\,&$b$\,&$b$\,&~$b$\\
\end{tabular}
&
\begin{tabular}{cccc}
$b$&\,$\infty$ & \,$b$& ~\,$b$\\
     $\infty$&\,$b$& $a$&~\,$b$\\
     $b$&$a$&$1$&~\,$b$\\
     $b$&\,$b$&\,$b$&~\,$b$\\
\end{tabular}
\\ \hline
\end{tabular}
\end{table}

\begin{cor}
If $S$ is an ai-semiring in $\{S_{(4,84)},S_{(4,94)},S_{(4,117)},S_{(4,173)},S_{(4,282)}\}$,
then $S$ is nonfinitely based.
\end{cor}
\begin{proof}
It is easily verified that $S_7$ can be embedded into $S$ and so $\mathsf{V}(S)$ contains $S_c(abc)$.
Let $n\geq2$, and let $\varphi: X\to S$ be an arbitrary substitution.
If $\varphi(t_{\mathbb H_n})\in\{\infty, 1\}$,
then $\varphi(t_{\mathbb H_n})=\varphi(t_{\mathbb H_n}^2)$
and so $\varphi(t_{\mathbb H_n})=\varphi(t_{\mathbb H_n}+t_{\mathbb H_n}^2)$.
If $\varphi(t_{\mathbb H_n})=a$ then $\varphi(x_ux_vx_w)=a$ for all hyperedges $\{u,v,w\}\in E$,
since $a$ is a minimal element in $(S, +)$.
As the only multiplicative decomposition of $a$ of length $3$
are of the form $1\cdot 1\cdot a=a\cdot 1\cdot 1=1 \cdot a\cdot 1$,
it follows that $|\varphi^{-1}(a)\cap e|=1$ for all $e\in E$
and so $\mathbb H_n$ is $2$-colourable, which is a contradiction. So $\varphi(t_{\mathbb H_n})\neq a$.
If $\varphi(t_{\mathbb H_n})=b$,
then $S\in\{S_{(4,94)},S_{(4,117)}\}$ and $\varphi(x_u)\varphi(x_v)\varphi(x_w)=b$ for some $\{u,v,w\}\in E$.
This implies that $(\varphi(x_u),\varphi(x_v),\varphi(x_w))\in\{(b,1,1),(1,1,b)\}$.
Since $x_vx_ux_w$ and $x_ux_wx_v$ are also summands of $t_{\mathbb H_n}$ and since $1\cdot b\cdot 1=\infty$,
we deduce that $\varphi(t_{\mathbb H_n})=\infty$, a contradiction.
So $\varphi(t_{\mathbb H_n})\neq b$.
Thus $S$ satisfies the identity \eqref{id112500}, with choice \eqref{id112501} for~$\bw_n$,
and so  $S$ is nonfinitely based by Theorem~\ref{thm24112501}.
\end{proof}

\begin{cor}\label{cor242212525}
The $4$-element ai-semiring $S_{(4,123)}$ is nonfinitely based.
\end{cor}
\begin{proof}
It is easy to see that $S_7$ can be embedded into $S_{(4,123)}$
and so $\mathsf{V}(S_{(4,123)})$ contains $S_c(abc)$.
Let $n\geq2$, and let $\varphi: X \to S_{(4,123)}$ be an arbitrary substitution.
If $\varphi(t_{\mathbb{H}_n})=\infty$, then $\varphi(t_{\mathbb{H}_n}+x_{v})=\infty$,
since~$\infty$ is the maximum element in $(S_{(4,123)}, +)$.
If $\varphi(t_{\mathbb{H}_n})=1$, then it is easy to see that $\varphi(x_v)=1$ for all $v\in V$.
This implies that $\varphi(x_{v})=1$ and so $\varphi(t_{\mathbb{H}_n}+x_{v})=1$.
If $\varphi(t_{\mathbb{H}_n})=a$, then $\varphi(x_{v_1})\varphi(x_{v_2})\varphi(x_{v_3})=a$ for all $\{v_1, v_2, v_3\}\in E$.
Since $a$ can only be expressed as the product of one $a$ and two $1$,
one can deduce that $\mathbb{H}_n$ is $2$-colourable, a contradiction.
So $\varphi(t_{\mathbb{H}_n}) \neq a$.
If $\varphi(t_{\mathbb{H}_n})=b$, then
$\varphi(x_{v_1})\varphi(x_{v_2})\varphi(x_{v_3})=b$ or $1$ for all $\{v_1, v_2, v_3\}\in E$.
This implies that $\varphi(x_{v})=b$ or $1$.
So $\varphi(t_{\mathbb{H}_n}+x_{v})=b$.
This shows that $S_{(4,123)}$ satisfies the identity \eqref{id112500} with choice \eqref{id112502} for~$\bw_n$.
By Theorem~\ref{thm24112501} we conclude that $S_{(4,123)}$ is nonfinitely based.
\end{proof}

\begin{cor}
The $4$-element ai-semiring $S_{(4,359)}$ is nonfinitely based.
\end{cor}
\begin{proof}
Following the notation in \cite{zrc},
let $S_{53}$ denote a $3$-element ai-semiring
whose Cayley tables for addition and multiplication given by Table \ref{tb24011601}.
It is a routine matter to verify that
both $S_{(4,123)}$ and $S_{(4,359)}$ are isomorphic to subdirect products of $S_7$ and $S_{53}$.
So ${\mathsf V}(S_{(4,359)})={\mathsf V}(S_{(4,123)})$.
By Corollary \ref{cor242212525} we deduce that $S_{(4,359)}$ is nonfinitely based.
\end{proof}

\begin{table}[ht]
\caption{The Cayley tables of $S_{53}$} \label{tb24011601}
\begin{tabular}{c|ccc}
$+$      &$\infty$&$a$&$1$\\
\hline
$\infty$ &$\infty$&$\infty$&$\infty$\\
$a$      &$\infty$&$a$&$a$\\
$1$      &$\infty$&$a$&$1$\\
\end{tabular}\qquad
\begin{tabular}{c|ccc}
$\cdot$  &$\infty$&$a$&$1$\\
\hline
$\infty$ &$\infty$&$\infty$&$\infty$\\
$a$      &$\infty$&$\infty$&$a$\\
$1$      &$\infty$&$a$&$1$\\
\end{tabular}
\end{table}

\begin{remark}
By using the same approach of Corollary \ref{cor24112501}, one can show that
every variety in the interval $[\mathsf{V}(S_c(abc)), {\mathsf V}(S_{(4, k)})]$
is nonfinitely based, where $k=84, 94, 117, 123, 173, 282, 359$.
\end{remark}

\section{The finite basis problem for subvarieties of $\mathsf{V}(S_7)$}
In this section we answer the finite basis problem for all subvarieties of the variety $\mathsf{V}(S_7)$.
In the course of this exploration, we also derive some basic results on the full subvariety lattice of $\mathsf{V}(S_7)$,
though the main results concerning this lattice will be given in Section~\ref{sec:blockhypergraph}.

By Corollary \ref{cor24112501} we know that each subvariety of $\mathsf{V}(S_7)$ that contains $S_c(abc)$
is nonfinitely based.
So it is enough to describe the subvarieties $\mathsf{V}(S_7)$ that do not contain $S_c(abc)$.
The following lemma, which is about the basic properties of $S_7$, will be used without explicit reference.

\begin{lem}\label{s7prop}
\hspace*{\fill}
\begin{itemize}
\item [$(a)$] $S_7$ satisfies the following identities\up:
\begin{align}
x^3 & \approx x^2; \label{id1}\\
xy  & \approx yx;    \label{id2}\\
x+xy& \approx xy^2;\label{id3}\\
x+y^2& \approx x^2y^2.   \label{id4}
\end{align}

\item [$(b)$] If $S$ is a member of $\mathsf{V}(S_7)$ and
$E(S)$ denotes the set of all multiplicative idempotents of $S$, then $E(S)$ is equal to $\{a^2 \mid a\in S\}$
and forms a subalgebra of $S$.

\item [$(c)$] A member of $\mathsf{V}(S_7)$ is subdirectly irreducible
if and only if it is a flat semiring that has a least nonzero multiplicative ideal.

\item [$(d)$] $\mathsf{V}(S_7)$ contains $S_c(a_1\cdots a_k)$ for all $k\geq 1$.
\end{itemize}
\end{lem}
\begin{proof}
$(a)$ It is easy to check this result.

$(b)$ This follows from the fact that $S$ satisfies the identities (\ref{id1}), (\ref{id2}) and (\ref{id4}).

$(c)$ It follows from Lemma \ref{lem24121301} immediately.

$(d)$ This follows from \cite[Proposition 2.6]{jrz}.
\end{proof}

We now identify the minimal nontrivial subvarieties of $\mathsf{V}(S_7)$.
\begin{pro}\label{nnp1}
$\mathsf{V}(M_2)$ and $\mathsf{V}(S_c(a))$ are the only minimal nontrivial subvarieties of $\mathsf{V}(S_7)$.
\end{pro}
\begin{proof}
By \cite[Theorem 1.1]{polin} and \cite{sr} we know that
$\mathsf{V}(L_2)$, $\mathsf{V}(R_2)$, $\mathsf{V}(M_2)$, $\mathsf{V}(D_2)$, $\mathsf{V}(N_2)$ and $\mathsf{V}(S_c(a))$
is a complete list of the minimal nontrivial subvarieties of the variety of all ai-semirings.
It is easy to see that both $M_2$ and $S_c(a)$ can be embedded into $S_7$.
On the other hand, none of $L_2$, $R_2$, $D_2$ and $N_2$ satisfy the identity (\ref{id4}).
So $\mathsf{V}(M_2)$ and $\mathsf{V}(S_c(a))$ are the only minimal nontrivial subvarieties of $\mathsf{V}(S_7)$.
\end{proof}

\begin{pro}\label{pro24100301}
Let $\mathcal{V}$ be a subvariety of $\mathsf{V}(S_7)$.
Then $\mathcal{V}$ does not contain $S_c(a)$
if and only if $\mathcal{V}$ satisfies the identity
\begin{equation}\label{id24101701}
x^2\approx x.
\end{equation}
\end{pro}
\begin{proof}
Suppose that $\mathcal{V}$ satisfies~\eqref{id24101701}.
It is easy to see that $S_c(a)$ does not satisfy~\eqref{id24101701}
and so $S_c(a)$ is not contained in $\mathcal{V}$.
Conversely, assume that $\mathcal{V}$ does not satisfy~\eqref{id24101701}.
Then there exists a semiring $S$ in $\mathcal{V}$ such that $a^2\neq a$ for some $a\in S$.
Since the identities~\eqref{id1} and~\eqref{id4} are satisfied by $S_7$,
it follows that $\{a^2, a\}$ is isomorphic to $S_c(a)$.
Thus $\mathcal{V}$ contains $S_c(a)$ as required.
\end{proof}

\begin{pro}\label{pro24100302}
Let $\mathcal{V}$ be a subvariety of $\mathsf{V}(S_7)$.
Then $\mathcal{V}$ does not contain $M_2$
if and only if $\mathcal{V}$ satisfies the identity
\begin{equation}\label{id24101702}
x^2y\approx x^2.
\end{equation}
\end{pro}
\begin{proof}
Suppose that $\mathcal{V}$ satisfies the identity (\ref{id24101702}). Since this identity
does not hold in $M_2$, it follows immediately that $\mathcal{V}$ does not contain $M_2$.
Conversely, assume that~$\mathcal{V}$ does not satisfy the identity (\ref{id24101702}).
Then there exists $S$ in $\mathcal{V}$ such that $a^2b\neq a^2$ for some $a, b \in S$.
By using the identities (\ref{id1}), (\ref{id2}) and (\ref{id3}),
it is easy to verify that $\{a^2, a^2b^2\}$ is isomorphic to $M_2$.
Hence $\mathcal{V}$ contains $M_2$.
\end{proof}

Let $\mathsf{V}(M_2, S_c(a))$ denote the variety generated by $M_2$ and $S_c(a)$.
From \cite[Theorem 3.8]{sr}
we know that the lattice $\mathcal{L}(\mathsf{V}(M_2, S_c(a)))$ of subvarieties of $\mathsf{V}(M_2, S_c(a))$
contains exactly $4$ varieties: $\mathsf{V}(M_2, S_c(a))$, $\mathsf{V}(M_2)$, $\mathsf{V}(S_c(a))$
and the trivial variety ${\bf T}$, which are all finitely based (see \cite{sr}).
To study the lattice $\mathcal{L}(\mathsf{V}(S_7))$ of subvarieties of $\mathsf{V}(S_7)$,
it is natural to consider the mapping
\[
\varphi: \mathcal{L}(\mathsf{V}(S_7)) \to \mathcal{L}(\mathsf{V}(M_2, S_c(a))),
~\mathcal{V} \mapsto \mathcal{V} \cap \mathsf{V}(M_2, S_c(a)).
\]
Then $\varphi$ is surjective and so $\mathcal{L}(\mathsf{V}(S_7))$ is the disjoint union of
$\varphi^{-1}({\bf T})$, $\varphi^{-1}(\mathsf{V}(M_2))$, $\varphi^{-1}(\mathsf{V}(S_c(a)))$
and $\varphi^{-1}(\mathsf{V}(M_2, S_c(a)))$.

\begin{pro}\label{pro24101840}
\hspace*{\fill}
\begin{itemize}
\item[$(1)$] $\varphi^{-1}({\bf T})=\{{\bf T}\}$.

\item[$(2)$] $\varphi^{-1}(\mathsf{V}(M_2))=\{\mathsf{V}(M_2)\}$.

\item[$(3)$] $\varphi^{-1}(\mathsf{V}(S_c(a)))=[\mathsf{V}(S_c(a)), \mathbf{N}]$,
where $\mathbf{N}$ denotes the subvariety of $\mathsf{V}(S_7)$
determined by the identity $(\ref{id24101702})$.

\item[$(4)$] $\varphi^{-1}(\mathsf{V}(M_2, S_c(a)))=[\mathsf{V}(M_2, S_c(a)), \mathsf{V}(S_7)]$.
\end{itemize}
\end{pro}
\begin{proof}
$(1)$ This follows from Proposition \ref{nnp1} immediately.

$(2)$
It is easy to see that $\mathsf{V}(M_2)$ is a member of $\varphi^{-1}(\mathsf{V}(M_2))$.
Now let $\mathcal{V}$ be an arbitrary variety in $\varphi^{-1}(\mathsf{V}(M_2))$.
Then $\mathcal{V}$ contains $M_2$, but does not contain $S_c(a)$.
By Proposition \ref{pro24100301} we have that ${\bf V}$ satisfies (\ref{id24101701}).
Since the identity (\ref{id4}) holds in $S_7$,
it follows that ${\bf V}$ satisfies $x+y\approx xy$, which is an equational basis of $\mathsf{V}(M_2)$.
Thus $\mathcal{V}=\mathsf{V}(M_2)$ and so $\varphi^{-1}({\bf M})=\{\mathsf{V}(M_2)\}$.

$(3)$ Let $\mathcal{V}$ be a variety in $[\mathsf{V}(S_c(a)), \mathbf{N}]$.
Then $\mathcal{V}$ contains $S_c(a)$ and satisfies the identity (\ref{id24101702}).
By Proposition \ref{pro24100302} we have that $\mathcal{V}$ does not contain $M_2$
and so $\varphi(\mathcal{V})=\mathsf{V}(S_c(a))$.
Conversely, let $\mathcal{V}$ be a variety in $\varphi^{-1}(\mathsf{V}(S_c(a)))$.
Then $\varphi(\mathcal{V})=\mathsf{V}(S_c(a))$
and so $\mathcal{V}$ contains $S_c(a)$, but does not contain $M_2$.
By Proposition \ref{pro24100302} again we deduce that $\mathcal{V}$ lies in $[\mathsf{V}(S_c(a), \mathbf{N}]$.

$(4)$ This is trivial.
\end{proof}

The following result provides some initial information about subdirectly irreducible members of $\mathsf{V}(S_7)$.
\begin{thm}\label{thm24101201}
Let $S$ be a nontrivial subdirectly irreducible member of $\mathsf{V}(S_7)$
and $I$ denote its least nonzero multiplicative ideal.
Then $I$ contains exactly two elements and $E(S)$ contains at most two elements.
If $|E(S)|=1$, then $S$ is a member of ${\bf N}$.
If $|E(S)|=2$, then $S$ either is isomorphic to $M_2$ or contains a copy of $S_7$.
\end{thm}
\begin{proof}
By Lemma \ref{s7prop} we know that $S$ is a flat semiring with $E(S)=\{a^2\mid a\in S\}$.
Let $e,f\in E(S)\setminus\{0\}$ and $a\in I\setminus\{0\}$.
Then $a\in Se\cap Sf$ and so $a=s_1e=s_2f$ for some $s_1,s_2\in S$.
This implies that $ae=af=a\neq 0$. Since $(S,\cdot)$ is $0$-cancellative,
it follows immediately that $e=f$ and so $E(S)$ contains at most two elements.

If $|E(S)|=1$, then $E(S)=\{0\}$ and so $a^2=0$ for all $a\in S$.
This shows that $S$ satisfies the identity (\ref{id24101702}) and so $S$ is a member of ${\bf N}$.
Suppose by way of contradiction that $|I|\neq 2$.
Then $|I|> 2$ and so there exist $a, b\in I\setminus\{0\}$ such that $a\neq b$.
Furthermore, we have that $S^1a=S^1b=I$ and so $a=cb$ and $b=da$ for some $c,d\in S$.
This implies that $a=cda$ and so $a=(cd)^2a=0a=0$, a contradiction.
Thus $|I|=2$.

If $|E(S)|=2$, then $E(S)=\{0,e\}$ for some $e \in S\backslash \{0\}$.
Suppose that $a^2=e$ for some $a\in S$. Then $ae=aa^2=a^3=a^2=e\neq 0$ and so $a=e$, since $(S,\cdot)$ is $0$-cancellative.
So $a^2=0$ for all $a\in S\setminus\{e\}$. Suppose that $ab=e$ for some $a,b\in S$. Then $e=e^2=(ab)^2=a^2b^2$ and so $a=b=e$.
We have shown that $e=ab$ implies that $e=a=b$ for all $a,b\in S$.
Now consider the following two cases.

{\bf Case 1}. $I = E(S)$. Suppose that $S\neq E(S)$. Take $a$ in $S\setminus E(S)$.
Then $E(S)=I\subseteq S^1a$ and so $e=sa$ for some $s\in S$, a contradiction.
Thus $S=E(S)$ and so $S$ is isomorphic to $M_2$.

{\bf Case 2}. $I\neq E(S)$. Let $a$ be an arbitrary element in $I$.
Since $I\subseteq Se$, it follows that there exists $s\in S$ such that $a=se$
and so $ea=ese=se=a$. This shows that $ea=a$ for all $a\in I$.
Now it is easy to see that $\{0,e,a\}$ is isomorphic to $S_7$ for all $a\in I\setminus E(S)$.
So $S$ contains a copy of $S_7$.
In the remainder we need only show that $|I|=2$.
Suppose that $a$ and $b$ are distinct elements of $I\setminus\{0\}$.
Then $S^1a=S^1b=I$ and so $a=cb$ and $b=da$ for some $c,d\in S\setminus\{0,e\}$.
This implies that $a=cda$ and so
$a=(cd)^2a=c^2d^2a=0\cdot a=0$, a contradiction.
Hence $|I|=2$ as required.
\end{proof}
Theorem \ref{thm24101201} will be revisited in Section~\ref{sec:blockhypergraph}, where the classification of subdirectly irreducible members of $S_7$ is given greater precision.
\begin{remark}\label{re24101801}
The result in Theorem \ref{thm24101201} is also true if $S$ is a subdirectly
irreducible flat semiring satisfying the identities (\ref{id1}) and (\ref{id2}).
This can be obtained from the proof of Theorem \ref{thm24101201}.
\end{remark}
\begin{pro}
Let $\mathcal{V}$ be a subvariety of $\mathsf{V}(S_7)$, and let $k\geq 1$ be a natural number.
Then $\mathcal{V}$ does not contain $S_c(a_1\cdots a_k)$
if and only if it satisfies the identity
\begin{equation}\label{24101201}
x_1\cdots x_k\approx (x_1\cdots x_k)^2.
\end{equation}
\end{pro}
\begin{proof}
Suppose that $\mathcal{V}$ satisfies the identity (\ref{24101201}).
It is easy to see that $S_c(a_1\cdots a_k)$ does not satisfy (\ref{24101201}), so $\mathcal{V}$ does not contain $S_c(a_1a_2\cdots a_k)$.
Conversely, assume that $\mathcal{V}$ does not satisfy (\ref{24101201}).
Then there exists a finite subdirectly irreducible flat semiring $S$ in $\mathcal{V}$ that does not satisfy (\ref{24101201}).
By Theorem \ref{thm24101201} we know that $S$ is a member of ${\bf N}$,
or $S$ is isomorphic to $M_2$, or $S$ contains a copy of $S_7$.

If~$S$ is isomorphic to $M_2$, then $S$ satisfies (\ref{24101201}), a contradiction.
If~$S$ contains a copy of $S_7$, then $\mathcal{V}$ contains $S_c(a_1\cdots a_k)$,
since $\mathsf{V}(S_7)$ contains $S_c(a_1\cdots a_k)$.
If~$S$ is a member of ${\bf N}$, then there exist pairwise distinct elements
$b_1, b_2, \ldots, b_k$ in $S$ such that $b_1b_2\cdots b_k \neq 0$.
Furthermore, it is a routine matter to verify that
the subsemiring of~$S$ generated by $\{b_1, b_2, \ldots, b_k\}$ is isomorphic to
$S_c(a_1\cdots a_k)$.
Therefore, $\mathcal{V}$ contains $S_c(a_1\cdots a_k)$ as required.
\end{proof}

As a corollary, we have the following result.
\begin{cor}\label{cor241018001}
Let $\mathcal{V}$ be a subvariety of ${\bf N}$,
and let $k\geq 1$ be a natural number. Then $\mathcal{V}$ does not contain
$S_c(a_1\cdots a_k)$ if and only if it satisfies the identity
\begin{equation}\label{id24101703}
x_1\cdots x_k\approx y_1\cdots y_k.
\end{equation}
\end{cor}

Let $k\geq 1$ be a natural number and let ${\bf N}_k$ denote the subvariety of ${\bf N}$
determined by the identity (\ref{id24101703}).
It is easy to see that $S_c(a_1a_2\cdots a_{k})$ is a member of ${\bf N}_{k+1}$,
but does not lie in ${\bf N}_k$. So the sequence
\[{\bf N}_1\subset {\bf N}_2\subset {\bf N}_3\subset\cdots{\bf N}_k\subset {\bf N}_{k+1}\subset \cdots\]
is an infinite strictly ascending chain in the interval $[{\bf N}_2, {\bf N}]$.

\begin{pro}\label{p3}
Every finite member of ${\bf N}$ lies in ${\bf N}_k$ for some $k\geq 1$,
that is, the multiplicative reduct of each finite member
of ${\bf N}$ is nilpotent.
\end{pro}
\begin{proof}
Let $S$ be a finite member of ${\bf N}$.
Then $(S, \cdot)$ is a finite nil-semigroup.
It is well-known that every finite nil-semigroup is nilpotent.
Thus $(S, \cdot)$ is nilpotent and so $S$ is a member of ${\bf N}_k$ for some $k\geq 1$.
\end{proof}

By Proposition \ref{p3} we immediately deduce the following.
\begin{cor}\label{c3.1}
The variety ${\bf N}$ is the join of all ${\bf N}_k$, $k\geq1$.
\end{cor}

\begin{cor}
The variety ${\bf N}$ is not finitely generated.
\end{cor}
\begin{proof}
Suppose that ${\bf N}$ is generated by a finite ai-semiring $S$.
Then by Proposition~\ref{p3},~$S$ lies in some ${\bf N}_k$.
This implies that ${\bf N}={\bf N}_k$, a contradiction. So ${\bf N}$ is not finitely generated.
\end{proof}

The full proof of the following proposition is postponed until the next section, where it is shown that ${\bf N}$ is the join of the varieties $\mathsf{V}(S_c(a_1\cdots a_k))$ for all $k\geq 1$.
\begin{pro}\label{pro24101901}
$[\mathsf{V}(S_c(a)), {\bf N}]$ is the disjoint union of $\{{\bf N}\}$ and the following intervals
\[
[\mathsf{V}(S_c(a_1\cdots a_k)), {\bf N}_{k+1}],~k \geq 1.
\]
\end{pro}
\begin{proof}
Let $\mathcal{V}$ be an arbitrary variety in $[\mathsf{V}(S_c(a)), {\bf N}]$.
If $\mathcal{V}$ contains $S_c(a_1\cdots a_k)$ for all $k \geq 1$,
then we show in Lemma~\ref{lem:wordstratification}, below, that $\mathcal{V}={\bf N}$.
Otherwise, there exists $k\geq 1$ such that $\mathcal{V}$ contains $S_c(a_1\cdots a_k)$,
but does not contain $S_c(a_1\cdots a_{k+1})$.
By Corollary \ref{cor241018001} we immediately deduce that $\mathcal{V}$ lies in
$[\mathsf{V}(S_c(a_1\cdots a_k)), {\bf N}_{k+1}]$.
\end{proof}

\begin{cor}\label{co241019001}
Let $\mathcal{V}$ be a variety in $[\mathsf{V}(S_c(a)), {\bf N}]$.
If $\mathcal{V}$ does not contain $S_c(abc)$,
then $\mathcal{V}$ is equal to $\mathsf{V}(S_c(ab))$ or $\mathsf{V}(S_c(a))$.
\end{cor}
\begin{proof}
Suppose that $\mathcal{V}$ is a variety in $[\mathsf{V}(S_c(a)), {\bf N}]$ that does not contain $S_c(abc)$.
Then by Proposition \ref{pro24101901}, the variety $\mathcal{V}$ lies in
$[\mathsf{V}(S_c(a_1)), {\bf N}_{2}]$ or $[\mathsf{V}(S_c(a_1a_2)), {\bf N}_{3}]$.
Combining \cite[Propositions 3.2 and 3.4]{rjzl} we deduce that
$\mathsf{V}(S_c(a_1))={\bf N}_{2}$ and $\mathsf{V}(S_c(a_1a_2))={\bf N}_{3}$.
So $\mathcal{V}$ is equal to $\mathsf{V}(S_c(ab)$ or $\mathsf{V}(S_c(a))$.
\end{proof}

Let $\mathcal{V}_1$ and $\mathcal{V}_2$ be ai-semiring varieties.
We denote by $\mathcal{V}_1\vee \mathcal{V}_2$ the join of $\mathcal{V}_1$ and $\mathcal{V}_2$,
that is, the smallest variety containing $\mathcal{V}_1$ and $\mathcal{V}_2$.
\begin{pro}\label{pro24101811}
$[\mathsf{V}(M_2, S_c(a)), \mathsf{V}(S_7)]$ is the union of
\[
\{\mathsf{V}(S_7), \mathsf{V}(M_2, S_c(a)), \mathsf{V}(M_2, S_c(ab))\}
\]
and
\[
[\mathsf{V}(M_2, S_c(abc)), \mathsf{V}(M_2)\vee{\bf N}].
\]
\end{pro}
\begin{proof}
Let $\mathcal{V}$ be a proper subvariety of $\mathsf{V}(S_7)$ that contains $M_2$ and $S_c(a)$.
By Theorem \ref{thm24101201} it follows that each subdirectly irreducible member of $\mathcal{V}$
is isomorphic to $M_2$ or lies in ${\bf N}$.
This implies that $\mathcal{V}$ is the join of $\mathsf{V}(M_2)$ and some member of $[\mathsf{V}(S_c(a)), {\bf N}]$.
By Corollary \ref{co241019001} it follows that $\mathcal{V}$ lies in
$[\mathsf{V}(M_2, S_c(abc)), \mathsf{V}(M_2)\vee{\bf N}]$
or is equal to $\mathsf{V}(M_2, S_c(a))$ or $\mathsf{V}(M_2, S_c(ab))$.
\end{proof}

\begin{pro}\label{3lem1}
The join of $\mathsf{V}(M_2)$ and ${\bf N}$ is the unique maximal subvariety of $\mathsf{V}(S_7)$.
\end{pro}
\begin{proof}
It is easy to see that
the identity $x^2+y\approx x^2y$ holds in $\mathsf{V}(M_2) \vee {\bf N}$,
but is not satisfied by $\mathsf{V}(S_7)$. This shows that $\mathsf{V}(M_2) \vee {\bf N}$ is
a proper subvariety of $\mathsf{V}(S_7)$. On the other hand, it follows from Theorem \ref{thm24101201} that
every proper subvariety of $\mathsf{V}(S_7)$ is contained in $\mathsf{V}(M_2) \vee {\bf N}$.
Thus ${\bf M}\vee {\bf N}$ is the unique maximal subvariety of $\mathsf{V}(S_7)$.
\end{proof}

\begin{pro}\label{pro24101830}
$\mathsf{V}(M_2, S_c(ab))$ is finitely based.
More precisely, $\mathsf{V}(M_2, S_c(ab))$ is the subvariety of the variety ${\bf F}$
determined by the identities $(\ref{id1})$, $(\ref{id2})$ and
\begin{equation}\label{id24101850}
x_1x_2x_3 \approx x_1^2+x_2^2+x_3^2.
\end{equation}
\end{pro}
\begin{proof}
It is easy to check that both $M_2$ and $S_c(ab)$ satisfy the identities~\eqref{id1},~\eqref{id2} and~\eqref{id24101850}.
In the remainder we shall show that every subdirectly irreducible flat semiring that satisfies \eqref{id1}, \eqref{id2} and~\eqref{id24101850} either is isomorphic to $M_2$ or lies in $\mathsf{V}(S_c(ab))$.
Let~$S$ be such an algebra. Then by Remark~\ref{re24101801} $S$ satisfies~\eqref{id24101702}, or is isomorphic to $M_2$,
or contains a copy of $S_7$.  The possibility of containing a copy of $S_7$ is eliminated because $S_7$ fails  law~\eqref{id24101850}.
If $S$ is isomorphic to $M_2$, then it is obvious that $S$ lies in $\mathsf{V}(M_2, S_c(ab))$.
If $S$ satisfies~\eqref{id24101702}, then by \cite[Proposition~3.2]{rjzl}~$S$ is a member of $\mathsf{V}(S_c(ab))$
and so it lies in $\mathsf{V}(M_2, S_c(ab))$.
This proves the required result.
\end{proof}

The following theorem, which is the main result of this section,
answers the finite basis problem for all subvarieties of $\mathsf{V}(S_7)$.
\begin{thm}
$\mathsf{V}(S_7)$ has exactly $6$ finitely based subvarieties:
the trivial variety ${\bf T}$, $\mathsf{V}(M_2)$, $\mathsf{V}(S_c(a))$, $\mathsf{V}(S_c(ab))$,
$\mathsf{V}(M_2, S_c(a))$ and $\mathsf{V}(M_2, S_c(ab))$.
\end{thm}
\begin{proof}
Let $\mathcal{V}$ be an arbitrary subvariety of $\mathsf{V}(S_7)$. If $\mathcal{V}$ contains $S_c(abc)$, then by Corollary \ref{cor24112501} $\mathcal{V}$ is nonfinitely based.
If $\mathcal{V}$ does not contain $S_c(abc)$, then
it follows from
Proposition \ref{pro24101840}, Corollary \ref{co241019001} and Proposition \ref{pro24101811} that
$\mathcal{V}$ is equal to one of the varieties ${\bf T}$, $\mathsf{V}(M_2)$, $\mathsf{V}(S_c(a))$, $\mathsf{V}(S_c(ab))$,
$\mathsf{V}(M_2, S_c(a))$ and $\mathsf{V}(M_2, S_c(ab))$.
From \cite{rjzl, sr} and Proposition \ref{pro24101830}
we know that these $6$ varieties are all finitely based.
This completes the proof.
\end{proof}

As a consequence, we have the following corollaries.
\begin{cor}
Let $\mathcal{V}$ be a subvariety of $\mathsf{V}(S_7)$.
Then $\mathcal{V}$ is finitely based if and only if $\mathcal{V}$ does not contain $S_c(abc)$.
\end{cor}

\begin{cor}\label{cor25011801}
$\mathsf{V}(S_c(abc))$ is the unique minimal nonfinitely based subvariety of $\mathsf{V}(S_7)$.
\end{cor}

\begin{remark}
Corollary \ref{cor25011801} improves \cite[Theorem 3.6]{rjzl},
which states that the variety $\mathsf{V}(S_c(abc))$ is a minimal nonfinitely based subvariety of $\mathsf{V}(S_7)$.
\end{remark}

\setlength{\unitlength}{0.8cm}
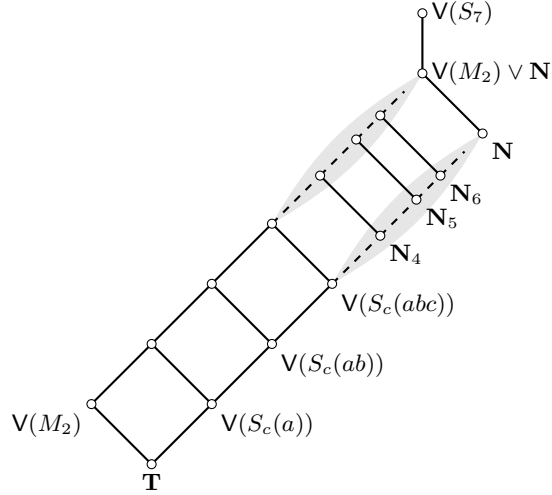
\begin{figure}[ht]
\begin{tikzpicture}[every node/.style={font=\small}, scale=0.8]

\path [fill = gray!20] (7,3) [out = 60,in = 210] to (9.5,5.5)  [out = -120,in = 30] to (7,3);
\path [fill = gray!20] (6,4) [out = 60,in = 210] to (8.5,6.5)  [out = -120,in = 30] to (6,4);

\foreach \x/\y in { 
        4/0, 5/1, 6/2, 7/3, 7.8/3.8, 8.4/4.4, 8.8/4.8, 9.5/5.5
    } {
        \draw[thick] (\x,\y) -- (\x-1,\y+1);
    }

\foreach \x/\y in { 
        3/1, 4/0, 4/2, 5/1, 6/2, 5/3,
    } {
        \draw[thick] (\x,\y) -- (\x+1,\y+1);
    }

\foreach \x/\y/\xx/\yy in { 
        7/3/7.8/3.8, 7.8/3.8/8.4/4.4, 8.4/4.4/8.8/4.8, 8.8/4.8/9.2/5.2, 6/4/6.8/4.8, 6.8/4.8/7.4/5.4, 7.4/5.4/7.8/5.8, 7.8/5.8/8.2/6.2
    } {
        \draw [thick, dashed] (\x,\y) to (\xx,\yy);
    }
 \draw[thick] (8.5,6.5) -- (8.5,7.5); 

    \foreach \x/\y/\label in {
        5/1/$\mathsf{V}(S_c(a))$,
        6/2/$\mathsf{V}(S_c(ab))$,
        7/3/$\mathsf{V}(S_c(abc))$,
        9.5/5.5/${\bf N}$,
        4/2/,
        5/3/,
        6/4/
    } {
        \node[circle, draw, fill=white, inner sep=1.2pt] at (\x,\y) {};
        \node[below right] at (\x,\y) {\label};
    }
    \node[circle, draw, fill=white, inner sep=1.2pt] at (3,1) {};
        \node[below left] at (3,1) {$\mathsf{V}(M_2)$};
    \node[circle, draw, fill=white, inner sep=1.2pt] at (4,0) {};
        \node[below] at (4,0) {$\mathbf{T}$};
    \node[circle, draw, fill=white, inner sep=1.2pt] at (8.5,6.5) {};
        \node[right] at (8.5,6.5) {$\mathsf{V}(M_2)\vee{\bf N}$};
    \node[circle, draw, fill=white, inner sep=1.2pt] at (8.5,7.5) {};
        \node[right] at (8.5,7.5) {$\mathsf{V}(S_7)$};

\foreach \x/\y/\label in {
        7.8/3.8/${\bf N}_4$,
        8.4/4.4/${\bf N}_5$,
        8.8/4.8/${\bf N}_6$,
        6.8/4.8/,
        7.4/5.4/,
        7.8/5.8/
    } {
        \node[circle, draw, fill=white, inner sep=1.2pt] at (\x,\y) {};
        \node[below right] at (\x,\y) {\label};
    }

\end{tikzpicture}
%
%
%
%
%
%
%
%
%
%
%
%
%
\caption{The subvariety lattice of $\mathsf{V}(S_7)$.}\label{figure1}
\end{figure}

The reader may find it useful to consult Figure~\ref{figure1}, which gives a schematic of the subvariety lattice of $\mathsf{V}(S_7)$.  The statements around the shaded region will be established in the next section, but the remaining structure has been established at this point.
In the figure, solid lines indicate the usual covering relation while dashed lines indicate variety containment.
We will show that the two shaded intervals are isomorphic and have both height and width equal to~$2^{\aleph_0}$:
moreover, this is true for each of the intervals $[\mathsf{V}(S_c(a_1\dots a_k)), {\bf N}_{k+1}]$, $k \geq 3$.
We will also show that the sequence of varieties
\[
\mathsf{V}(S_c(a_1a_2a_3))\subset \mathsf{V}(S_c(a_1a_2a_3a_4))\subset \cdots \subset \mathsf{V}(S_c(a_1a_2\cdots a_k))\subset \cdots
\]
provides a different stratification of~${\bf N}$, not shown in Figure~\ref{figure1}.

\section{Partition systems and Kneser hypergraphs}\label{sec:blockhypergraph}
Following Jackson and Lee~\cite{jaclee}, an algebra is said to be \emph{of type $2^{\aleph_0}$} if it generates a variety with $2^{\aleph_0}$ subvarieties.
It is perhaps surprising that a finite algebra can be of type~$2^{\aleph_0}$, but there are now many examples known, and the property is trivially inherited by any algebra whose variety contains an algebra of type ~$2^{\aleph_0}$.
Within semigroups and monoids the property has received particular attention, starting with Trahtman~\cite{tra};
two very recent contributions are Glasson~\cite{gla} and Gusev~\cite{gus}, who show that $M(abba)$ and $M(aabb)$, respectively are of type $2^{\aleph_0}$ in the monoid signature (these are the noncommutative versions of the construction, and the results concern the multiplicative reduct of the definition in the current paper).
The only finite semiring example that has been established so far is in Ren et al.~\cite[Example~5.7]{rjzl} where it is shown that the ai-semiring $M(abacdc)$ generates a semiring variety with continuum many subvarieties.
In this section we significantly enhance this result by proving that $S_7$ (which is isomorphic to $M(a)$, as a semiring) has type $2^{\aleph_0}$, as well as a number of  properties of the variety lattice in the shaded region of Figure~\ref{figure1}.

Our approach again reinforces the strong connection with hypergraphs and combinatorial properties.  We introduce the notion of a ``block hypergraph'', which is closely related to the notion of a Kneser hypergraph,
and which turns out to precisely capture subdirectly irreducible algebras in $\mathsf{V}(S_7)$.
Recall that the \emph{$r$-regular Kneser hypergraph} $\KG^r(n,k)$ is the hypergraph consisting of all $k$-element subsets of $\{1,\dots,n\}$ with hyperedges consisting of all $r$-sets of pairwise disjoint $k$-element subsets.
Kneser graphs ($r=2$) have been very heavily studied since Lov\'asz's resolution of Kneser's conjecture~\cite{lov}, showing that the chromatic number of $\KG^2(n,k)$ is precisely $n-2k+2$ (provided $n\geq 2k$, so that there is at least one edge).
Various extensions of Lov\'asz's result were established by Alon, Frankl and Lov\'asz~\cite{AFL} for hypergraphs and other combinatorial objects.

Consider a finite set $I$, and a family $\mathscr{F}$ of nonempty subsets of $I$: in other words, $(I, \mathscr{F})$ is any hypergraph.
We can form a new hypergraph $\mathbb{H}_\mathscr{F}$ from $\mathscr{F}$ by taking the vertices to be the elements of $\mathscr{F}$ (the hyperedges of $(I, \mathscr{F})$) and defining the hyperedges of $\mathbb{H}_\mathscr{F}$ to be those sets $\{A_1, \dots, A_k\}$ whenever $A_1, \dots, A_k\in\mathscr{F}$ are pairwise disjoint with union equal to $I$.
In other words, the hyperedges correspond to partitionings of $I$ by members of $\mathscr{F}$.
From the semiring perspective we will primarily be interested in the particular case where
\begin{enumerate}
\item[(Ha)] every $A\in \mathscr{F}$ lies within a hyperedge of $\mathbb{H}_\mathscr{F}$, and
\item[(Hb)] no element of $\mathscr{F}$ is the disjoint union of a set of smaller members of $\mathscr{F}$.
\end{enumerate}
If (Ha) and (Hb) hold then we refer to $\mathscr{F}$ as a \emph{partition system}, though the conditions imposed are really only conveniences: the property (Ha) will ensure that a certain construction is subdirectly irreducible (but a semiring is still definable) and the property (Hb) ensures that the elements of $\mathscr{F}$ are a minimal generating set (but a semiring is still definable and  the ``indecomposable'' elements of $\mathscr{F}$ will be a minimal generating set).
We could also allow infinite versions of the block hypergraph, requiring finite partitions of infinite sets, but we do not need it, as we are considering locally finite varieties, which are then generated by their finite subdirectly irreducible members.
A hypergraph isomorphic to the hypergraph $\mathbb{H}_\mathscr{F}$ formed from the blocks of a partition system $\mathscr{F}$ in this way will be called a \emph{block hypergraph}.
\begin{example}\label{eg:kl}
For any $k,\ell\geq 1$, the family $\mathscr{F}_{k,\ell}$ of $k$-element subsets of $\{1,2,\dots,k\ell\}$ is a partition system in which every hyperedge in $\mathbb{H}_{\mathscr{F}_{k,\ell}}$ has cardinality~$\ell$.  This hypergraph coincides with the \emph{Kneser hypergraph} $\KG^\ell(k\ell,k)$.
\end{example}
Note that the block hypergraph $\mathbb{H}_{\mathscr{F}_{1,\ell}}$ is the hypergraph consisting of a single $\ell$-element hyperedge.
The partition system $\mathscr{F}_{2,\ell}$ is just the complete graph on $2\ell$ vertices, but the block hypergraph $\mathbb{H}_{\mathscr{F}_{2,\ell}}$ formed over it is an $\ell$-uniform hypergraph.
When $\ell=2$ then $\mathbb{H}_{\mathscr{F}_{2,2}}$ is just a disjoint union of single edges, but when $\ell=3$ the hypergraph then the hypergraph is more complicated.

Let $\mathscr{F}$ be a partition system,
and let $\overline{\mathscr{F}}$ denote the closure of $\mathscr{F}$ under taking disjoint unions $\cup_{1\leq i\leq k}A_i$ of members of $\mathscr{F}$,
provided that $\{A_1,\dots,A_k\}$ is a subset of a hyperedge of $\mathbb{H}_\mathscr{F}$.
So, a union $\cup_{1\leq i\leq k}A_i$ is in $\overline{\mathscr{F}}$ provided there are $A_{k+1},\dots,A_\ell$ such that $A_{1},\dots,A_\ell$  partition~$I$.
We may define a flat semiring $S_{\mathbb{H}_\mathscr{F}}$ on $\overline{\mathscr{F}}\cup\{\infty\}$ by giving $\infty$ its usual additive and multiplicative properties in a flat semiring, and defining multiplication between elements of $\overline{\mathscr{F}}$ to be disjoint union, if it sits within~$\overline{\mathscr{F}}$, or $\infty$ otherwise:
\[
A\cdot B:=\begin{cases}
A\cup B&\text{ if $A\cap B=\varnothing$ and $A\cup B\in \overline{\mathscr{F}}$}\\
\infty&\text{ otherwise}.
\end{cases}
\]
The reader can verify directly that this produces a flat semiring, however it follows from Lemma~\ref{lem:subsemi}.
The restriction to disjoint unions that can be extended to a full hyperedge is not strictly necessary either, and arises as an intermediate case in the proof, however the semiring obtained is not subdirectly irreducible in this case.
\begin{lem}\label{lem:subsemi}
For any finite nonempty set $I$, and any partition system $\mathscr{F}$ on $I$, the semiring $S_{\mathbb{H}_\mathscr{F}}$ is isomorphic to a quotient of a subsemiring of $S_c(a_1a_2\cdots a_{|I|})$.
\end{lem}
\begin{proof}
Let ${\bf w}$ denote the word $a_1a_2\cdots a_{|I|}$.  It is useful to rename the elements of~$I$ as the set $\{a_1,\dots,a_{|I|}\}$ and for any subset $A$ of $I$, let ${\bf w}_A$ denote $\prod_{a_i\in A}a_i$, a subword of ${\bf w}$ (so, a redundancy is that ${\bf w}={\bf w}_I$).
Then ${\bf w}_{A_1}\cdot {\bf w}_{A_2}\cdot\dots\cdot{\bf w}_{A_k}\neq \infty$ if and only if the family $A_1,\dots,A_k$ are pairwise disjoint, in which case it equals ${\bf w}_{A_1\cup\dots \cup A_k}$.
Let $T_\mathscr{F}$ be the subsemiring generated by the words ${\bf w}_A$ for $A\in \mathscr{F}$.  (The semiring $T_\mathscr{F}$ corresponds to the alternative definition that does not restrict to disjoint unions that can be extended to a full hyperedge.)
Then $S_{\mathbb{H}_\mathscr{F}}$ arises by collapsing (to $\infty$) all elements that do not multiplicatively divide~${\bf w}$.
A product $\prod_{1\leq i\leq k}{\bf w}_{A_i}$ remains uncollapsed precisely when there are $A_{k+1},\dots,A_\ell$ such that $A_{1},\dots,A_\ell$ are a partition of $I$ (so that $(\prod_{1\leq i\leq k}{\bf w}_{A_i})\cdot (\prod_{k+1\leq i\leq \ell}{\bf w}_{A_i}) = {\bf w}$), matching exactly the definition of non-$\infty$ multiplication in $S_{\mathbb{H}_\mathscr{F}}$ by way of disjoint unions within $\overline{\mathscr{F}}$.
\end{proof}
\begin{remark}
In the definition of  a hypergraph semiring in Section~\ref{sec:intro}, the restriction to $3$-uniform hypergraphs of girth at least $5$ enables a uniform description of the multiplication between generators: in particular, relating to conditions~(4) and~(5) of the definition.
In~\cite{jrz}, the definition also allows for $k$-uniform hypergraphs, for $k\geq 3$, though the girth restriction remains.
The semiring~$S_{\mathbb{H}_\mathscr{F}}$ is in a natural way the hypergraph semiring of the block hypergraph $\mathbb{H}_\mathscr{F}$, though now there may be variation to the size of the hyperedges, and there is no girth requirement.  Instead of conditions (4) and (5), the definition of multiplication is instead managed by the concrete representation in terms of partition systems and Lemma~\ref{lem:subsemi}.
In the case that a block hypergraph is $k$-uniform and girth at least $5$ (such as if it is a tree or a forest), then it is not hard to see that the two definitions coincide.
\end{remark}

A \emph{transversal} of a hypergraph is a subset of the vertices that intersects each hyperedge exactly once.
Let $\mathscr{T}$ be a family of transversals for a hypergraph $\mathbb{H}$.
Borrowing from point set topology terminology, we say that $\mathbb{H}$ is~\emph{$T_0$ with respect to~$\mathscr{T}$} if every vertex is contained in at least one transversal from $\mathscr{T}$, and for every pair of vertices $u\neq v$ there is a transversal in $\mathscr{T}$ selecting $u$ but not $v$ or a transversal selecting $v$ but not $u$.
The hypergraph $\mathbb{H}$ is said to be \emph{transversal complete} (with respect to $\mathscr{T}$) if whenever $\{v_1,\dots,v_k\}$ is a set of vertices with the property that every transversal from $\mathscr{T}$ intersects $\{v_1,\dots,v_k\}$ at precisely one vertex, then $\{v_1,\dots,v_k\}$ is a hyperedge.
The following lemma can be easily adjusted to accommodate the case where isolated vertices are allowed, but we will not need it.
\begin{lem}\label{lem:transversal}
The following are equivalent for a hypergraph $\mathbb{H}$ without isolated vertices.
\begin{itemize}
\item[$(1)$] $\mathbb{H}$ is a block hypergraph\up;

\item[$(2)$] $\mathbb{H}$ is $T_0$ and is transversal complete with respect to some family $\mathscr{T}$ of transversals.
\end{itemize}
When the conditions hold, then the hypergraph $\mathbb{H}$ is isomorphic to $\mathbb{H}_\mathscr{F}$ for a partition system $\mathscr{F}$ on the set $\mathscr{T}$.
\end{lem}
\begin{proof}
If $\mathbb{H}$ is isomorphic to a hypergraph on a partition system $\mathscr{F}$ over a set $I$, then for every $i\in I$, the set of all members of $\mathscr{F}$ containing $i$ is a transversal.  Let $\mathscr{T}$ denote this family of transversals---one for each $i\in I$.  As every element $A$ of $\mathscr{F}$ is nonempty, there is always a choice of $i$ that will place~$A$ in at least one transversal.  Moreover, for every $A\neq B$ in $\mathscr{F}$ there is a point~$i$ in~$A$ but not~$B$ (or~$B$ but not~$A$), and then selecting this point will place $A$ into the transversal but not $B$ (or vice versa).  So $\mathbb{H}$ is $T_0$ with respect to transversals from $\mathscr{T}$.  The definition of hyperedges on $\mathscr{F}$ ensures that $\mathbb{H}$ is transversal complete with respect to $\mathscr{T}$.

Conversely, if $\mathbb{H}$ is $T_0$ and transversal complete with respect to some family of transversals $\mathscr{T}$, then let $I$ denote the set $\mathscr{T}$ itself,  and represent each vertex $v$ in~$\mathbb{H}$ as the set $T_v\subseteq \mathscr{T}$ of transversals that include $v$.
Then the set of vertices of~$\mathbb{H}$ is represented as a set of nonempty subsets of $I$, which we can denote by $\mathscr{F}$.  Let $e=\{v_1,\dots,v_k\}$ be a hyperedge of~$\mathbb{H}$.
The definition of a transversal ensures that the sets $T_{v_1},\dots,T_{v_k}$ form a partition of $I$ so that the mapping $v\mapsto T_v$ is a hypergraph homomorphism.
The $T_0$-separability ensures that this homomorphism is a bijection, and the transversal completeness of~$\mathbb{H}$ with respect to $\mathscr{T}$ ensures that it is also a bijection on hyperedges, hence an isomorphism.
\end{proof}
It can be seen from the proof that if every transversal containing~$u$ also contains~$v$, then $T_u\subseteq T_v$, so if the~$T_0$ separation property is replaced by the stronger $T_1$ condition (for every $u\neq v$ there is a transversal containing $u$ but not $v$), then this (and transversal completeness) is equivalent to block hypergraphs in which no block is a subset of any other block.
It is well known that it is \texttt{NP}-complete to determine if a hypergraph has a transversal (the positive 1-in-3SAT problem is a particular case of this, in this case of 3-uniform hypergraphs).
In \cite{jac:SAT} it is shown that it is \texttt{NP}-hard to distinguish those hypergraphs that have no transversal at all, from those that are $T_1$ (stronger than $T_0$) with respect to transversals and transversal complete; this is true even in the 3-uniform hypergraph setting.
From Lemma~\ref{lem:transversal} it follows in particular that it is \texttt{NP}-hard to determine if a hypergraph is a block hypergraph.
\begin{remark}\label{rem:transversal}
In general the transversal family $\mathscr{T}$ can serve as a certificate for verifying that a hypergraph $\mathbb{H}$ is a block hypergraph.
If $V$ denotes the vertex set of $\mathbb{H}$, then at most $\binom{|V|}{2}$ transversals are required to witness the $T_0$ separation property (and any larger set of transversals will preserve $T_0$-separability).
There may be exponentially many non-hyperedges however, and it must be verified that there is a transversal in $\mathscr{T}$ that selects either none or more than one vertex.
\end{remark}

We may now revisit Theorem~\ref{thm24101201}.
We first recall some constructions relating to flat extensions of partial algebras.
For a partial algebra $B$ and element $p\in B$, we say that $b\in B$ \emph{divides} $p$ if there is a term $t(x,y_1,\dots)$ and elements $c_1,\dots$ such that $t^B(b,c_1,\dots)$ is defined and is equal to $p$.
The degenerate case where $y_1,\dots$ is empty is allowed,
and is important because it is often required to show that $p$ divides itself: using the term~$x$.
Let $B_p$ denote the partial algebra obtained by taking all elements of~$B$ that divide~$p$, and restricting the existing partial operations to these elements.
(This might be called a \emph{divisor restriction} of~$B$.)
For a class $K$ of partial algebras of the same signature, the class operator $\mathsf{D}(K)$ will denote the class of partial algebras of the form $B_p$ for $B\in K$ and $p\in B$.
In the case of a single finite partial algebra $P$, it is shown in Willard~\cite[Theorem~1.2]{wil} that the subdirectly irreducible algebras in $\mathsf{V}(\flat(P))$ are, up to isomorphism, the flat extensions of partial algebras in $\mathsf{DSP}(P)$, where $\mathsf{P}$ denotes direct powers of $P$,~$\mathsf{S}$ denotes sub (partial) algebras.
Jackson~\cite[Theorem~3.3]{jac:flat} extends this to the case of classes of partial algebras, not necessarily finite: if $K$ is a class of partial algebras in the same signature, and $\flat(K)$ denotes the class consisting of the flat extensions of members of $K$, then the subdirectly irreducible members of $\mathsf{V}(\flat(K))$, up to isomorphism, are the flat extensions of members of $\mathsf{DSPP}_{\rm u}(K)$, where $\mathsf{P}_{\rm u}$ denotes ultraproducts (and $\mathsf{P}$ is for direct products over nonempty classes of partial algebras).

The first two items of the following theorem replicate the first two items of Theorem~\ref{thm24101201} with a different proof,
while the third item is a refinement.
\begin{thm}\label{thm:siins7}
Let $S$ be a finite subdirectly irreducible member of $\mathsf{V}(S_7)$.  Then precisely one of the following statements holds\up:
\begin{itemize}
\item[$(1)$] $S$ is isomorphic to $M_2$\up;
\item[$(2)$] $S$ contains a copy of a $S_7$ as a subsemiring, and hence generates the same variety as $S_7$\up;
\item[$(3)$] $S$ is isomorphic to $S_\mathbb{H}$ for some block hypergraph $\mathbb{H}$.
\end{itemize}
In the case of \up($2$\up), $S$ is isomorphic to $S^1_\mathbb{H}$ for some block hypergraph $\mathbb{H}$.
\end{thm}
\begin{proof}
As $S_7$ is finite, it follows from \cite[Theorem 1.2]{wil} that every subdirectly irreducible algebra in $\mathsf{V}(S_7)$ is isomorphic to one of the form $\flat(B_p)$,
where $B$ is a subalgebra of $\{1, a\}^I$ (for some nonempty index set $I$) and $p$ is an element of $\{1, a\}^I$.

We now consider some mutually exclusive cases.  If $p$ is the constant tuple $\underline{1}$, then $B_p=\{\underline{1}\}$ and we arrive at $M_2\cong \flat(B_p)$.  Now assume that $p$ has at least one coordinate equal to $a$.
Let $J$ be the largest subset of $I$ on which $p$ is constantly~$a$, and note that if $b$ divides $p$ then $b$ is constantly $1$ on $I\backslash J$.
Thus there is no loss of generality in assuming that $I=J$, and that $p$ is the constant tuple $\underline{a}$.  If $\underline{1}\in B_p$, then $\{\underline{1},\underline{a},\infty\}$ is a subalgebra of $\flat(B_p)$ that is isomorphic to $S_7$.
Now assume that $\underline{1}\notin B_p$, or has been removed for analysis of the remaining part of $B_p$ (the element~$\underline{1}$ can be removed because $B\backslash\{\underline{1}\}$ is also a subalgebra of $S_7^I$).
For each element $b\in B_p$, let $I_b$ denote the subset $\{i\in I\mid b(i)=a\}$ of $I$.
Observe that if $b,c\in B_p$, then $b\cdot c$ is defined if and only if  $I_b\cap I_c=\varnothing$, in which case $I_{b\cdot c}=I_b\cup I_c$.
So non-$\infty$ products can be viewed as disjoint unions of sets of the form $I_c$.
Let~$P$ denote those elements of $B_p$ that do not arise as a nontrivial product.
Then the family $\mathscr{F}=\{I_b\mid b\in P\}$ is a partition system, and $\flat(B_p)$ is isomorphic to the block hypergraph semiring over~$\mathscr{F}$.
\end{proof}

An alternative version  of Corollary~\ref{c3.1} is now possible. This lemma also completes the missing detail of the proof of Proposition~\ref{pro24101901}.
\begin{lem}\label{lem:wordstratification}
The variety ${\bf N}$ is generated by $S_c(a_1\cdots a_k)$, $k\geq 1$.
\end{lem}
\begin{proof}
The variety ${\bf N}$ is generated by its finite subdirectly irreducibles, which are block hypergraph semirings by Theorem~\ref{thm:siins7}.  By Lemma~\ref{lem:subsemi} each such block hypergraph semiring is isomorphic to a quotient of a subsemiring of $S_c(a_1\cdots a_k)$ for some $k$.
\end{proof}
The proof of Lemma~\ref{lem:wordstratification} used the fact that every finite subdirectly irreducible member $S$ of ${\bf N}$ lies in the variety generated by $S_c(a_1\cdots a_k)$ for some $k$.
It is easily seen that an upper bound, in terms of $|S|$, on the smallest choice of $k$ can derived from Lemmas~\ref{lem:transversal} and~\ref{lem:subsemi}, though we do not use this here.

The following lemma does not assume that $\mathcal{V}$ is a subvariety of $\mathsf{V}(S_7)$, though it is implicit that the top element is a multiplicative zero.
\begin{lem}\label{lem:moresi}
Let $\mathcal{V}$ be a variety generated by a class of flat semirings.
Then the finite subdirectly irreducible members of $\mathcal{V}\vee \mathsf{V}(M_2)$ are those of $\mathcal{V}$ along with $M_2$.
\end{lem}
\begin{proof}
Let $K$ denote the class of all subdirectly irreducible members of $\mathcal{V}$.
We need to prove that if $S$ is a  subdirectly irreducible algebra in $\mathcal{V}\vee \mathsf{V}({M}_2)$ that is not isomorphic to ${M}_2$, then it is in $K$.
The variety $\mathcal{V}\vee \mathsf{V}({M}_2)$ is generated by $K\cup\{{M}_2\}$.
Let $K'$ be the family of partial algebras obtained from members of $K$ by removing the element~$\infty$ and keeping only the operation $\cdot$, now only partially defined.  Removing~$\infty$ from ${M}_2$ simply produces the one-element algebra ${\bf 1}$ with total multiplication $\cdot$.
It follows from \cite[Theorem~3.3]{jac:flat} that
$S$ is isomorphic to the flat extension of a partial algebra ${\bf P}$ in $\mathsf{DSPP}_{\rm u}(K'\cup\{{\bf 1}\})$.
As ${\bf 1}$ is the trivial one-element total algebra,
it is routine to see that the isomorphism closure of $\mathsf{DSPP}_{\rm u}(K'\cup\{{\bf 1}\})$ is equal to the isomorphism closure of $\{{\bf 1}\}\cup\mathsf{DSPP}_{\rm u}(K')$.
As $S$ is not isomorphic to $M_2$,
it follows that $S$ is isomorphic to the flat extension of a partial algebra in $\mathsf{DSPP}_{\rm u}(K')$, and therefore lies in $\mathcal{V}$, as required.
\end{proof}
An immediate corollary of Lemma~\ref{lem:moresi} is the following.
\begin{cor}\label{cor:directproduct}
The subvariety lattice of $\mathsf{V}({M}_2)\vee {\bf N}$
is isomorphic to the direct product of the $2$-element lattice and the subvariety lattice of
${\bf N}$\up: the interval $[{\bf T},{\bf N}]$ is isomorphic to $[\mathsf{V}({M}_2),\mathsf{V}({M}_2)\vee {\bf N}]$ under the map $\mathcal{V}\mapsto \mathsf{V}({M}_2)\vee \mathcal{V}$.
\end{cor}

We consider the term ${\bf t}_\mathbb{H}$ as in Section~\ref{sec:NFB}, now allowing for the variable hyperedge size in a hypergraph $\mathbb{H}$:
\[
{\bf t}_\mathbb{H} := \sum_{\{v_1,\dots,v_k\}\in E_\mathbb{H}}x_{v_1}\cdots x_{v_k}
\]
where $E_\mathbb{H}$ denotes the hyperedge set of $\mathbb{H}$ and
$\{x_v\mid v\in V_\mathbb{H}\}$ is a copy of the vertex set $V_\mathbb{H}$ of $\mathbb{H}$.
\begin{lem}\label{lem:hyperhomom}
Let $\mathscr{F}_1$ and $\mathscr{F}_2$ be partition systems on sets $I$ and $J$ respectively.  Let $\mathbb{G}$ and $\mathbb{H}$ denote the corresponding block hypergraphs.  If $\mathbb{G}$ and $\mathbb{H}$ are $k$-uniform for some $k>2$, then $S_{\mathscr{F}_2}\models \bt_{\mathbb{G}}\approx \bt_{\mathbb{G}}^2$ if and only if there is no homomorphism from $\mathbb{G}$ to $\mathbb{H}$.
\end{lem}
\begin{proof}
Consider any assignment $\nu$ of the variables $\{x_v\mid v\in \mathscr{F}_1\}$ into $S_{\mathscr{F}_2}$.
Now each summand of $\bt_{\mathbb{H}}$ is a product of variables of length $k$, and $S_{\mathscr{F}_2}$ is $(k+1)$-nilpotent.
(The $k$-uniformity assumption is used here.)
It follows that $\nu(\bt_\mathbb{G})\in\{J,\infty\}$ and the law fails if and only if the output $J$ is possible.  Now $\nu(\bt_\mathbb{G})=J$ if and only if each  summand $x_{v_1}\dots x_{v_k}$ of $\bt_\mathbb{G}$ takes the value $J$ under $\nu$.
But this occurs precisely when~$\nu$ maps the variables to vertex elements and $\{\nu(x_{v_1}),\dots,\nu(x_{v_k})\}$ is a hyperedge for all hyperedges $\{v_1,\dots,v_k\}$ of $\mathbb{G}$.
This is precisely when there is a homomorphism from $\mathbb{G}$ to $\mathbb{H}$ (simply map the vertex $v$ of $\mathbb{G}$ to the vertex $\nu(x_{v})$ of $\mathbb{H}$).
\end{proof}
The assumption of $k$-uniformity can be dropped from Lemma~\ref{lem:hyperhomom}, but it requires using a quite broad and slightly technical generalisation of the notion of hypergraph homomorphism.  This will not be needed for our results, so is omitted.

Homomorphisms between Kneser hypergraphs turn out to be tightly connected to homomorphisms between Kneser graphs.  We note two facts from the literature.
\begin{lem}\label{lem:BBGR}
Let $r\geq 3$ and $n_1\geq rk_1$, $n_2\geq rk_2$.
\begin{enumerate}
\item[$(1)$] \up(Bonomo-Braberman et al.~\cite[Theorem~1]{BBDVPV}.\up) For $r\geq 3$, there is a homomorphism from $\KG^r(n_1,k_1)$ to $\KG^r(n_2,k_2)$ if and only if there is a homomorphism between the Kneser graphs $\KG^2(n_1,k_1)$ to $\KG^2(n_2,k_2)$.
\item[$(2)$] \up(Godsil and Royle~\cite[Lemma~7.9.3]{godroy}.\up) If $n_1>2k_1$ and $n_1/k_1=n_2/k_2$ then there is a homomorphism from $\KG^2(n_1,k_1)$ to $\KG^2(n_2,k_2)$ if and only if $k_1$ divides $k_2$.
\end{enumerate}
\end{lem}
A family of hypergraphs is \emph{homomorphism independent} if there are no homomorphisms between any two distinct members.
\begin{cor}\label{cor:continuumkneser}
Let $r>2$. Then the family of Kneser hypergraphs $\{\KG^r(rp,p)\mid p\text{ prime}\}$
is an infinite homomorphism independent family of $r$-regular block hypergraphs.
\end{cor}
\begin{proof}
By Lemma~\ref{lem:BBGR} part (1), it suffices to show that the family of Kneser graphs $\{\KG^2(rp,p)\mid p\text{ is prime}\}$ is homomorphism independent, noting that $r>2$ as required.  As $rp/p=r=rq/q$ for any two primes $p,q$, it follows from part (2) of Lemma~\ref{lem:BBGR} that there are no homomorphisms between any two distinct members of the family.
\end{proof}
\begin{cor}\label{cor:continuum}
Let $r>2$.
Then the interval $[\mathsf{V}(S_c(a_1\dots a_r)), {\bf N}_{r+1}]$ has cardinality of the continuum.
It contains both a chain and an antichain of cardinality~$2^{\aleph_0}$.
\end{cor}
\begin{proof}
Let $P$ denote the set of all prime numbers and for each $p\in P$ let $\mathbb{H}_p$ abbreviate $\mathbb{H}_{\mathscr{F}_{p,r}}$.
Then the semiring $S_{\mathscr{F}_{p,r}}$ is $(r+1)$-nilpotent, so sits within the interval $[\mathsf{V}(S_c(a_1\dots a_r)),{\bf N}_{r+1}]$.  Let $Q$ be any proper  subset of $P$ and let $p\in P\backslash Q$.
By Corollary~\ref{cor:continuumkneser} there is no homomorphism from $\mathbb{H}_p$ to any $\mathbb{H}_q$ for $q\in Q$, so $\{S_{\mathscr{F}_{q,r}}\mid q\in Q\}$ satisfies ${\bf t}_{\mathbb{H}_p}\approx {\bf t}_{\mathbb{H}_p}^2$ by Lemma~\ref{lem:hyperhomom}.
As ${\bf t}_{\mathbb{H}_p}\approx {\bf t}_{\mathbb{H}_p}^2$ fails on the semiring~$S_{\mathscr{F}_{p,r}}$, it follows that $S_{\mathscr{F}_{p,r}}\notin\mathsf{V}(\{S_{\mathscr{F}_{q,r}}\mid q\in Q\})$, and that all $2^{\aleph_0}$ subsets of $P$ lead to different subvarieties of ${\bf N}_{r+1}$.
The observation on width and height comes from well-known properties of the powerset lattice of the denumerable set $P$, which we have order-embedded within $[\mathsf{V}(S_c(a_1\dots a_r)),{\bf N}_{r+1}]$.
\end{proof}

The varieties ${\bf N}_1, {\bf N}_2, {\bf N}_3, \dots$ and
$\mathsf{V}(S_c(a_1)), \mathsf{V}(S_c(a_1a_2)), \mathsf{V}(S_c(a_1a_2a_3)), \dots$
provide two stratifications of ${\bf N}$ (by Corollary~\ref{c3.1} and Lemma~\ref{lem:wordstratification} respectively), but despite the decomposition in Proposition~\ref{pro24101901}, they are quite different.  The variety $\mathsf{V}(S_c(a_1a_2\dots a_{k}))$ is the smallest subvariety of ${\bf N}$  not contained in ${\bf N}_{k}$,
but it does not itself contain ${\bf N}_k$, nor equal ${\bf N}_{k+1}$ (except in the base cases of $k=1,2$ where they coincide).
We now show that no finite member of ${\bf N}$ generates a variety containing ${\bf N}_k$ for $k \geq 4$.
We first need some results on colouring.

All block hypergraphs are $2$-colourable, by Lemma~\ref{lem:transversal}, but homomorphisms relate to strong colourings more than to colourings.  The next proposition uses the inability to strongly $\ell$-colour $\mathbb{H}_{\mathscr{F}_{k,\ell}}$ (the $k$-subsets of $\{1,\dots,k\ell\}$, in the notation of Example~\ref{eg:kl}).
\begin{lem}\label{lem:kneser}
If there is a strong $n$-colouring of $\mathbb{H}_{\mathscr{F}_{k,\ell}}$ then $n\geq k\ell-2k+2$.
\end{lem}
\begin{proof}
Assume that $n\geq \ell$ is such that there is a strong $n$-colouring of $\mathbb{H}_{\mathscr{F}_{k,\ell}}$ (noting that $n<\ell$ is impossible).
It follows that no two disjoint $k$-sets in $\{1,\dots,k\ell\}$ are given the same colour, because every such pair extends to a hyperedge in $\mathbb{H}_{\mathscr{F}_{k,\ell}}$.
Thus we have an $n$-colouring of the Kneser graph $\KG^2(k\ell, k)$.
Lov\'asz~\cite{lov} proved that the chromatic number of $\KG^2(k\ell, k)$ is $k\ell-2k+2$, so it follows that $n\geq k\ell-2k+2$.
\end{proof}
\begin{pro}\label{pro:extra}
For any $n,\ell>2$ there exists $k$ such that $S_{\mathscr{F}_{k,\ell}}$ is an $(\ell+1)$-nilpotent hypergraph semiring that is not contained in the variety of any $n$-element member of ${\bf N}$.
In particular, for $\ell>2$, there is no finite member of ${\bf N}$ that generates a variety containing~${\bf N}_{\ell+1}$.
\end{pro}
\begin{proof}
Let $S$ be an $n$-element member of ${\bf N}$, and observe that $S$ satisfies $x^2\approx \infty$.
Let $\mathbb{H}$ abbreviate $\mathbb{H}_{\mathscr{F}_{k,\ell}}$ and
consider the identity $\bt_\mathbb{H}\approx \bt_\mathbb{H}^2$, which fails on $S_{\mathscr{F}_{k,\ell}}$.
We shall show that it holds in $S$.
Let $V$ be the vertices of $\mathbb{H}$ and consider an assignment~$\theta$ of the variables $\{x_v \mid v\in V\}$ into $S$.
Choose $k$ such that $n<k\ell-2k+2$, which is possible because $\ell>2$.
By Lemma~\ref{lem:kneser}, there is no strong $n$-colouring of $\mathbb{H}_{\mathscr{F}_{k,\ell}}$, so there is a hyperedge $\{v_1,\dots,v_\ell\}$ of $\mathbb{H}_{\mathscr{F}_{k,\ell}}$ such that $\theta(x_{v_1}),\dots,\theta(x_{v_\ell})$ contains a repeat.
Hence by $x^2\approx \infty$, the value of $\theta(x_{v_1}\dots x_{v_\ell})$ is $\infty$.
Hence both sides of $\bt_\mathbb{H}\approx \bt_\mathbb{H}^2$ take the value $\infty$, showing that the identity is satisfied.
It follows that~$S_{\mathscr{F}_{k,\ell}}$ is not in the variety generated by $S$.
\end{proof}

From \cite{sr} and \cite[Proposition 3.2]{rjzl} we know that
${\bf N}_1=\mathbf{T}$, ${\bf N}_2=\mathsf{V}(S_c(a_1))$ and ${\bf N}_3=\mathsf{V}(S_c(a_1a_2))$.
By Proposition \ref{pro:extra} we immediately deduce the following result.

\begin{cor}
Let $k\geq 1$ be a natural number.
Then ${\bf N}_k$ is finitely generated if and only if $k\leq 3$.
\end{cor}

\section{Conclusion}
We have completely answered the finite basis problem for ai-semirings that lie in the variety of $S_7$.
By contrast, we still do not know whether every finite ai-semiring whose variety contains $S_7$ is nonfinitely based.
Figure \ref{figure1} presents a basic framework of the subvariety lattice of $\mathsf{V}(S_7)$
and describes its lower parts.
The subvariety lattice of ${\bf N}$ is the disjoint union of $\{{\bf N}, {\bf T}\}$ and
the intervals of the form $[\mathsf{V}(S_c(a_1\dots a_k)),{\bf N}_{k+1}]$ for all $k\geq1$ (Proposition~\ref{pro24101901}), and we have shown that for $k=1,2$ these intervals are trivial (Corollary~\ref{co241019001}), while for each $k>2$ the interval has cardinality of the continuum, thereby showing that $S_7$ has type~$2^{\aleph_0}$.
This proves the conjecture proposed in \cite{rjzl}: ``it remains
plausible that $M(a)$ generates a semiring variety with continuum many subvarieties''.

We have shown that $({\bf N}_{k})_{k \geq 1}$ and $(S_c(a_1\dots a_k))_{k \geq 1}$
provide two different stratifications of ${\bf N}$.
Since
${\bf N}_1=\mathbf{T}$, ${\bf N}_2=\mathsf{V}(S_c(a_1))$ and ${\bf N}_3=\mathsf{V}(S_c(a_1a_2))$,
it follows from \cite[Proposition 3.4]{rjzl} that all of the intervals $[{\bf N}_1, {\bf N}_2]$ and $[{\bf N}_2, {\bf N}_3]$,
$[\mathsf{V}(S_c(a_1)), \mathsf{V}(S_c(a_1a_2))]$ and $[\mathsf{V}(S_c(a_1a_2)), \mathsf{V}(S_c(a_1a_2a_3))]$ contain $2$ varieties.
But the interval $[{\bf N}_3, {\bf N}_4]$ has cardinality of the continuum.
Indeed, let $\mathcal{V}$ be a variety in $[{\bf N}_3, {\bf N}_4]$ that do not contain $S_c(a_1a_2a_3)$.
Then by Corollary \ref{cor241018001} we have that $\mathcal{V}={\bf N}_3$.
Since ${\bf N}_3=\mathsf{V}(S_c(a_1a_2))$, it follows immediately that
$\mathsf{V}(S_c(a_1a_2a_3))$ lies in $[{\bf N}_3, {\bf N}_4]$ and so
$[{\bf N}_3, {\bf N}_4]$ is the linear sum of the singleton lattice $\{{\bf N}_3\}$ and the interval $[\mathsf{V}(S_c(a_1a_2a_3)), {\bf N}_4]$.
By Corollary \ref{cor:continuum} we obtain that $[{\bf N}_3, {\bf N}_4]$ is uncountable.
The intervals $[{\bf N}_{k+1}, {\bf N}_{k+2}]$ for $k\geq 3$ are at least slightly different, because of Proposition~\ref{pro:extra}, however by adapting the proof of Corollary \ref{cor:continuum} we can also show that they have cardinality $2^{\aleph_0}$.
To see this, observe that if we fix $r = k+1$, then the laws ${\bf t}_{\mathbb{H}_p}\approx {\bf t}_{\mathbb{H}_p}^2$ (for prime $p$) used in the proof of Corollary~\ref{cor:continuum} are satisfied by ${\bf N}_{k+1}$, so hold in the join
${\bf N}_{k+1}\vee\mathsf{V}(\{S_{\mathscr{F}_{q,r}}\mid q\in Q\})$ (a subvariety of ${\bf N}_{k+2}$) but not in $S_{\mathscr{F}_{p,r}}$.
However, we do not know the cardinality of the intervals
$[\mathsf{V}(S_c(a_1\dots a_k)), \mathsf{V}(S_c(a_1\dots a_{k+1}))]$ for $k \geq 3$.
It is quite possible that all of these intervals also contain $2^{\aleph_0}$ varieties.


\bibliographystyle{amsplain}

\end{document}